%% file: main-amsart.tex
\documentclass[a4paper,11pt]{amsart}
\input{preamble.tex}
\usepackage[left=3cm,right=3cm]{geometry}
\usepackage{mlmodern}

\begin{document}
\input{auth-addr.tex}

\maketitle

\input{sections/1-introduction.tex}
\input{sections/2-foa.tex}

\input{sections/3-foa.tex}
\input{sections/4-comparing.tex}

\bibliographystyle{amsalpha}
\bibliography{main}{}

\end{document}

%% file: preamble.tex
\usepackage{fouche/fouche}

\usepackage{mathabx}
\usepackage{xspace,colortbl}
\newcommand{\ptoo}{\rightsquigarrow}

\def\emdash{\text{-}}

\def\textEM{Ei\-len\-berg--Mo\-ore\@\xspace}
\let\too\longrightarrow

\def\slice{/\!\!/}

\NewDocumentCommand{\en}{ m O{} }{[#1,#1]_{#2}}
\def\ie{i.e.\@\xspace}

\NewDocumentCommand{\pairUp}{ O{r} m m }{
\ar[#1]_-{#3}
\ar@<.5em>@{<-}[#1]^-{#2}
}
\NewDocumentCommand{\pairDown}{ O{r} m m }{
\ar[#1]^-{#2}
\ar@<-.5em>@{<-}[#1]_-{#3}
}
\NewDocumentCommand{\pair}{ O{r} m m O{\perp} }{
\ar@<-.5em>@{<-}[#1]_-{#3}
\ar@{}[#1]|-{#4}
\ar@<.5em>[#1]^-{#2}
}
\NewDocumentCommand{\oppair}{ O{r} m m O{\perp} }{
\ar@<-.5em>[#1]_-{#3}
\ar@{}[#1]|-{#4}
\ar@{<-}@<.5em>[#1]^-{#2}
}
\NewDocumentCommand{\barealg}{ o m m O{} }{
	\IfNoValueTF{#1}
	{#2 \ltimes_{#4} #3 }
	{#2 \ltimes^{#1}_{#4} #3 }
}
\NewDocumentCommand{\emalg}{ m m }{\barealg[\sfE\sfM]{#1}{#2}}

\NewDocumentCommand{\barecoalg}{ o m m O{} }{
	\IfNoValueTF{#1}
	{#2 \oright_{#4} #3 }
	{#2 \oright^{#1}_{#4} #3 }
}

\def\falgFibre#1#2{\barealg{\{#1\}}#2}

\newcommand{\klalg}[2]{\barealg[\sfK\sfl]{#1}{#2}}
\newcommand{\coklalg}[2]{\barecoalg[\sfK\sfl]{#1}{#2}}
\NewDocumentCommand{\id}{ O{} }{\mathrm{id}_{#1}}

\newcommand{\ladj}[1]{#1_L}
\newcommand{\radj}[1]{#1_R}

\newcommand{\pointsTotal}[1]{\cate{Pt}(#1)}

\NewDocumentCommand{\triple}{ O{r} m m m }{
\ar[#1]|-{#3}
\ar@<-.5em>@{<-}[#1]_-{#4}
\ar@<.5em>@{<-}[#1]^-{#2}
}

\newcommand{\Poly}[1]{\cate{Poly}_{#1}}

\newcommand{\initial}{\emptyset}
\newcommand{\terminal}{\ast}

\newcommand{\defn}[1]{\emph{#1}} % For highlighting definitions
\def\cate#1{\mathsf{#1}}

\usepackage{doi}
\usepackage{enumitem}
\setlist[itemize,enumerate]{itemsep=0pt}

%% file: auth-addr.tex
\title{Fibrations of algebras}
\author[Ahman]{Danel Ahman}
\address{{Department of Mathematics}, {University of Tartu}}
\email{danel.ahman@ut.ee}
\email{ulo.reimaa@ut.ee}

\author[Coraglia]{Greta Coraglia}
\address{{Dipartimento di Filosofia}, {Università degli Studi di Milano}}
\email{greta.coraglia@unimi.it} 

\author[Castelnovo]{Davide Castelnovo}
\address{{Dipartimento di Matematica `Tullio Levi-Civita'}, {Università degli Studi di Padova}}
\email{davide.castelnovo@math.unipd.it}

\author[Loregian]{Fosco Loregian}
\address{{Department of Cybernetics}, {Tallinn University of Technology}}
\email{folore@taltech.ee}

\author[Martins-Ferreira]{{Nelson} {Martins-Ferreira}}
\address{{Politécnico de Leiria}}
\email{martins.ferreira@ipleiria.pt}

\author[Reimaa]{{\"Ulo} {Reimaa}}

\begin{abstract}
We study fibrations arising from indexed categories of the following form: fix two categories $\mathcal{A},\mathcal{X}$ and a functor $F :  \mathcal{A} \times \mathcal{X} \longrightarrow\mathcal{X} $, so that to each $F_A=F(A,-)$ one can associate a category of algebras $\mathbf{Alg}_\mathcal{X}(F_A)$ (or an \textEM, or a Kleisli category if each $F_A$ is a monad). We call the functor $\int^{\mathcal{A}}\mathbf{Alg} \to \mathcal{A}$, whose typical fibre over $A$ is the category $\mathbf{Alg}_\mathcal{X}(F_A)$, the \emph{fibration of algebras} obtained from $F$.
Examples of such constructions arise in disparate areas of mathematics, and are unified by the intuition that $\int^\mathcal{A}\mathbf{Alg} $ is a form of \emph{semidirect product} of the category $\mathcal{A}$, acting on $\mathcal{X}$, via the `representation' given by the functor $F : \mathcal{A} \times \mathcal{X} \longrightarrow\mathcal{X}$.
After presenting a range of examples and motivating said intuition, the present work focuses on comparing a generic fibration with a fibration of algebras: we prove that if $\mathcal{A}$ has an initial object, under very mild assumptions on a fibration $p : \mathcal{E}\longrightarrow \mathcal{A}$, we can define a canonical action of $\mathcal{A}$ letting it act on the fibre $\mathcal{E}_\varnothing$ over the initial object. This result bears some resemblance to the well-known fact that the fundamental group $\pi_1(B)$ of a base space acts naturally on the fibers $F_b = p^{-1}b$ of a fibration $p : E \to B$.
\end{abstract}

%% file: sections/1-introduction.tex
\section{Introduction}
The present paper is part of a project trying to build a general theory of \emph{parametrized functors}, studying how they arise in different parts of category theory and neighbouring subjects. A \emph{parametrized (endo)functor} is understood simply as a functor of type $F : \params\times\carriers \too \carriers$, so that for every fixed object $\param$ of a category $\params$ we have a functor $F(\param,\firstblank)=F_\param : \carriers \too \carriers$, and for every morphism $u : A\to A'$ in $\params$, a natural transformation between $F_A$ and $F_{A'}$.

For such a functor $F$, we call the category $\params$ the \emph{category of parameters} of $F$, while $\carriers$ is the \emph{category of carriers}. We can also view a parameterized endofunctor as a functor $\params \to \en\carriers$ into the category of endofunctors on $\carriers$. Replacing the latter category with the category of (co)pointed endofunctors, the category of (co)monads, etc., allows us to easily define parametrized \emph{pointed} functors, parametrized \emph{(co)monads}, etc.

In a certain sense -- which will be made precise below and in \autoref{subsec:act} -- a parametrized functor $F: \params\times\carriers \too \carriers$ can be viewed as a representation of $\params$ over $\carriers$; as such, it will be fruitful to consider invariants attached to $F$ that are inspired by representation theory.

One straightforward procedure associating to $F$ an invariant that keeps track of its action is the following: consider each functor $F_\param$ separately, but `glue' together all categories of $F_\param$-algebras (objects with an action of $F_A$, in an evident sense) in a \emph{bundle} (the total category of a fibration), indexed over the parameters $\params$, so that the fibre over $\param\in\params$ is precisely the category of $F_\param$-algebras, where objects are pairs $(X,\xi : F_\param X\to X)$ and morphisms $(X,\xi)\to(Y,\theta)$ are maps between the carriers $X,Y$ that are compatible with the structure maps $\xi,\theta$.

This total category is precisely the fibered category $\var[p^F]{\grothCont[\params]\Alg}\params$ obtained by applying the \emph{Grothendieck construction} to the pseudofunctor $\params^\op \too \Cat$ sending $\param\mapsto\Alg(F_\param)$.

Arguing similarly for parametrized (co)monads $T : \params\times\carriers\too\carriers$, and for the {pseudo}functor sending $T_\param$ to its (co)\textEM category, yields the \emph{(op)fi\-bra\-tion of \textEM (co)algebras} of $T$.

Although there might be other ways to attach to $F$ or $T$ a category describing their actions on $\carriers$, the fibred category so obtained (the `fibration of algebras' in the title) turns out to be defined canonically, to be very rich in structure, and quite descriptive of the action of $F$ (thanks to the equivalence prescribed by the Grothendieck construction).

In fact, besides the fibration $p^F$ itself, the category $\grothCont[\params]\Alg$ comes equipped with a `carrier' functor $V:\grothCont[\params]\Alg\too\carriers$ that forgets the $F_A$-algebra structure on an object, thus we have a span of functors%\footnote{}
\[\label{fundamental_span}
	\vxy{
	\carriers & \ar[l]_-V \grothCont[\params]\Alg \ar[r]^-{p^F} & \params.
	}\]
This diagram will induce, in turn, a canonical functor $\langle p^F,V\rangle : \grothCont[\params]\Alg \to \params\times\carriers$ by the universal property of products. The first notable result is that the analogous functor $\langle p^T,V\rangle : \grothCont[\params]\EM \to \params\times\carriers$ for a parametrized monad $T \colon \params \to \Mnd{\carriers}$ is monadic (cf. \autoref{proposition:_limits}), characterizing $\grothCont[\params]\EM$ as an \textEM object in the slice 2-category $\Cat/\params$. Similarly, if $F$ is just a parametrized functor, $\langle p^T,V\rangle$ presents $\grothCont[\params]\Alg$ as an \emph{inserter} (cf. \autoref{remark_on_endofu}) in the same 2-category.

Thus, the idea that the fibration of algebras of a parametrized functor $F$ can be seen as a representation of $\params$ over $\carriers$ is backed up by the fact that $\grothCont[\params]\Alg$ has a universal property, analogous to the action groupoid of a group representation.

% \ulo{In the following section, I need to change the endomorphism algebras to monad algebras. I really do object to calling a sequence $\bullet \xrightarrow{f} \bullet \xrightarrow{g} \bullet$ exact if we merely ask for $gf = 0$. I changed the assumptions so that we at least have $\clX$ as the kernel of $p$.}

Viewing a parametrized monad $T : \params \too \Mnd{\carriers}$ as an action of $\params$ on $\carriers$, we argue that the category  $\grothCont[\params]\EM$ can also be thought of as a kind of a semidirect product $\algTotal[T]\params\carriers$ of $\carriers$ and $\params$.

One justification for this is that under mild assumptions on $\params,\carriers$ (namely: the existence of a terminal object in $\carriers$, of an initial object in $\params$, and 
that $T_\initial$ is the identity monad), ensure that $V$ in \eqref{fundamental_span} has a left adjoint $\ladj{V}$, and $p^T$ has a right adjoint $p^T_R$. This gives us a diagram
% \[\label{quaglia}
% 	\vxy{
% 		\carriers \pair {}{V} & \grothCont[\params]\Alg \pair {p^F}{} & \params
% 	}
% \]
\[\label{quaglia}
	\vxy{
		\carriers \pair {\ladj{V}}{V} & \grothCont[\params]\EM \pair {p^T}{p^T_R} & \params
	}
\]
of left adjoints in which the composition $p^T\circ V_L$ `is zero' (=constant at the initial object of $\params$), which can be argued to be a quite well-behaved \emph{extension diagram}\footnote{Though this `extension perspective' of the present work will not be explored further here.}.

% This remark becomes more apparent if both $\clA,\clX$ are pointed (they admit a zero object): in this case, we have a diagram
% \[\label{mangusta}
% 	\vxy{
% 		1\ar[r] & \clX\ar[r] & \grothCont[\params]\Alg \ar[r] & \clA \ar[r] &1
% 	}
% \]
% which is a short exact sequence of left adjoints in the evident sense that each arrow is a left adjoint, and the composition of any two consecutive functors is the functor constant at zero.

% \ulo{The issue with putting $1$ at each end of the sequence is that we don't yet have a notion exactness that work for non-short exact sequences. (Yes, even if in abelian algebra people call sequences of that shape \textit{short exact sequences}, what they \textit{really} mean is that the the category of \textit{short exact sequences} can be embedded in the category of \textit{exact sequence} if we plop $1$ at either end. In our case we don't yet have such an embedding.)}

If furthermore $\carriers$ has an initial object, then $p^T \colon \grothCont[\params]\EM  \to \params$ admits a full and faithful left adjoint, which exhibits our extension as a kind of a \textit{split extension}. 
This motivates us to think of the fibration of algebras as a categorified semidirect product, leading us to adopt the notation
\[
	\textstyle
	\algTotal[T]\params\carriers \coloneqq \grothCont[\params]\EM 
	\quad \text{and} \quad
	\algTotal[F]\params\carriers \coloneqq \grothCont[\params]\Alg
\]
as syntactic sugar in both the monad and endofunctor algebra situations. We gather additional evidence of why this is a convenient and sound choice in \autoref{matters_of_fun}, and later in \autoref{ltimes_from_ltimes_mon}, but we start using this notation immediately.

A subsequent question we analyze is how many nice properties $\algTotal\params\carriers$ inherits, provided $\params,\carriers$ and all the functors $F_A$ are themselves well-behaved? Monadicity is quite useful in covering the case of existence of limits and colimits, namely: if limits exist in $\clA,\clX$ they exist in $\algTotal[T]\params\carriers$, and are created by $\langle p^T,V\rangle$. It similarly goes for $\algTotal[F]\params\carriers$ by $\langle p^F,V\rangle$, since all forgetful functors from inserter diagrams create limits. The existence of colimits is less automatic, but it can be reduced to the usual arguments about the existence of colimits in an \textEM category. %(such as the classical argument of Linton which uses reflexive coequalizers to construct all colimits).

%It is also reasonable to conjecture that examples of such construction are abundant and that they are instances of a general theory. This is the theory of `fibrations of algebras' that we expose in this note: 

%The theory of fibrations of algebras is the study of categories fibered over a base $\params$, of the form $\algTotal[F]\params\carriers$ for some parametrized functor $F : \clA\times\clX\too\clX$, and of all the variations on this theme.
%Sketching this general theory is the purpose of a longer work from various sub- and supersets of the same group of authors mentioned here, where

While there is a lot more to say on this topic, in the current work we only sketch general theory and focus on two fundamental aspects.
\begin{enumerate}
	\item We outline the essential lines of the theory, strictly to the extent that this is needed to tackle a specific problem; most notably,
	      \begin{result}[\autoref{yoink_a_2_fibration}]\label{result_1}
		      The correspondence $(\params,\carriers,F)\mapsto \algTotal[F]\params\carriers$ is a functor, defined from a domain that in a specific sense is of the form $\Cat \ltimes^\ell \Cat$, which is a 2-dimensional lax instance of a fibration of algebras.
	      \end{result}
	      \begin{result}\label{result_2}
		      Let $T$ be a parametrized monad. The category $\algTotal[T]\params\carriers$ is monadic over the product $\params\times\carriers$; this yields as a corollary, in one fell swoop, all the following:
		      \begin{itemize}
			      \item that limits in $\algTotal[T]\params\carriers$ are created as limits in $\params\times\carriers$, see \autoref{proposition:_limits};
			      \item as a consequence of a well-known result by Linton, that colimits are computed from reflexive coequalizers in the base, see \autoref{prop:fib_for_colim};
			      \item that there is a simple criterion, based on the adjoint functor theorem, to turn $p^T$ into a bifibration, see \autoref{swindle_theorem}.
		      \end{itemize}
	      \end{result}
	\item We fix a fibration $p : \clE\fib\params$ over $\params$, and under really minimal assumptions, find a `best-approximating' fibration of algebras receiving a canonical morphism of fibrations from $p$. Thus,
	      \begin{result}[\autoref{theorem:_the_fibration-monad_adjunction}]
		      Minimal assumptions ensure that every fibration can be converted into a fibration of algebras.
	      \end{result}
\end{enumerate}
%
%\ulo{Reminder: the following section mentiones the terminal object of the initial fibre. (Change this if this assumption is removed from the statement in section 4.)}
%
This second point constitutes the gist of the present paper; let us outline it more in detail: first of all, assume that $\params$ has an initial object. Then, given a fibration $p : \clE\fib\params$, we consider its fibre over the initial object, say $\clE_\varnothing$, and under the assumption that all its fibres have initial and terminal object, and some coproducts exist in its domain $\clE$ (we call such fibrations \emph{pruned}, the rationale being that the collection of stalks of the prestack associated to $p$ is a bit smaller and simpler to describe) we find a parametrized monad $T^p : \params\times\clE_\varnothing \too\clE_\varnothing$ associated to $p$, and we build a commutative triangle (actually, a morphism of fibrations) as below.
\[\vxy{
		\clE\ar[dr]_{p}\ar[rr]^{\eta_p} && \algTotal\params{\clE_\varnothing} \ar[dl]\\
		& \params &
	}\]
The monad $T^p$ has the property that $T^p_\varnothing$ is the identity functor. The construction we just sketched sets up a pair of adjoint 2-functors
\[\label{koala}
	\vxy{ \niceFib{\params} \pair{p \mapsto T^p}{p^T \mapsfrom T} & \ParMndInit{\params} \,,}
\]
(cf. \eqref{pruned_adjunction}) between the 2-categories of pruned fibrations over $\params$ and the 2-category of pruned parametrized monads, and the arrow $\eta^p$ is the unit of such adjunction.

A few remarks are in order to understand our presentation better. First of all, we focus on monads not only to anchor our discourse to a single thread (doing otherwise would unavoidably sacrifice clarity) but also because parametrized monads are more studied and already general enough to provide evidence that the theory we present is flexible and modular; monads are also more well-behaved: for bare endofunctors, the existence of free algebras, \ie of a left adjoint to the forgetful functor sending an $F_\param$-algebra to its carrier, does not come entirely for free.

If free algebras exist, a similar analysis could be carried out for parametrized functors alone, with some care; in \autoref{remark_on_endofu} we discuss how one of our main results, a monadicity theorem, can also be proved for generic parametrized endofunctors. It's a trade off between expressive power (an analogue of the theorem is true in more generality) and simplicity (the technology used is unavoidably more complex, viz we have to encode a universal property in a 2-categorical limit).

The importance of a `recognition principle' lies in the fact that presented with a generic fibration $p$, it might be hard to decide whether it is of the form $\algTotal[F]\params\carriers$; for example:
\begin{itemize}
  \item the \emph{simple fibration} described in \cite[1.3.3]{CLTT} is of use in \emph{simple type theory} (cf. \cite{Farmer2023}) and arises collating coKleisli categories in a fibration, $\int \mathrm{coKl} (A \times \firstblank)$: this was noted in \cite[Exercise 1.3.4.(ii)]{CLTT};
	\item it was noticed in \cite{Hedges2018} that \emph{lenses} arise instead from the fibrewise opposite\footnote{The fibrewise opposite fibration of $\var[p]{\grothCont F_p}\clB$ is the one associated to the pseudofunctor $\clB^\op\xto{F_p}\Cat\xto{\op}\Cat$.} fibration: $\int \mathrm{coKl} (A \times \firstblank)^\op$.
\end{itemize}
If the category of carriers and parameters coincide and admit finite limits:
\begin{itemize}
  \item the arrow category, corresponding to the codomain fibration, arises applying the Gro\-then\-dieck construction to the co\textEM algebras of the comonad $A\times\firstblank$ obtaining the category $\int \mathrm{coEM} (A \times \firstblank)$; this was noted in \cite[\textit{ibi}]{CLTT}.%, in an exercise;
	\item the coslice category, arises both as the fibrewise opposite fibration of co\textEM algebras $\int \mathrm{coEM} (X \times \firstblank)^\op$, and as the fibration of algebras for the constant polynomial endofunctor.% \cite[4.6.2]{notesonpoly}.
\end{itemize}
At least in the case of monads, we address this issue by giving a criterion in this respect: a (pruned) fibration $p$ is a fibration of \textEM algebras for a (pruned) parametrized monad if and only if the unit $\eta^p$ of the adjunction in \eqref{koala} is an equivalence of categories.

Compare the statement of \autoref{theorem:_the_fibration-monad_adjunction} with the following two results.
\begin{itemize}
	\item Let $p : E\fib B$ be a fibration between pointed spaces (so that $p$ has a `typical' fiber $F$); then, the fundamental group(oid) of the base $B$ acts in a canonical way on $F$.
	\item Let $A$ be an Abelian group, and $u : G\to A$ be an object of the slice category $\Grp/A$; then, $A$ acts on $\ker u$ in a canonical way.
\end{itemize}
This leaves open a few questions that we plan to expand on in subsequent work:
\begin{itemize}
	\item the work of \cite{gambino_kock_2013,Hedges2018}, expounded also in \cite[Ch. 4]{von_glehn_2015}, seems to suggest that when the fibration of algebras is a bifibration, the fibrewise opposite of the opfibration structure can also retain an explicit description and is of some interest even if it's not an (op)fibration of (co)algebras anymore. Is there a special way in which the duality involution of the 2-category $\Fib$ of fibrations acts on objects of the form $\algTotal\params\carriers$?
	\item The intuition
	% given in \eqref{mangusta}
	of $\algTotal\params\carriers$ as a semidirect product would require us to develop an abstract theory of extensions of categories, in a 2-category of adjunctions; besides having intrinsic interest regardless of its immediate application to fibrations of algebras, such a theory has to be written from scratch, as it is not present in the literature on 2-categories. This will require some effort in a dedicated project, and together with the previous problem, it will require a more thorough study of the 2-dimensional properties of fibrations of algebras.
	\item Every fibration $\var[p]\clE\clB$ induces a factorisation system on its total category, from the vertical/cartesian decomposition of each arrow $E\to E'$, and among all factorisation systems torsion theories \cite{jans1965,CHK,janelidze2007characterization} are connected particularly tightly with the theory of fibrations, see \cite{RT}. Recent work of Gran, Kadjo and Vercruysse \cite{GKV:hopf_torsion} finds a torsion theory in the category of cocommutative Hopf algebras, which for us is a motivating example, cf. \autoref{cartier_roba}. It is likely that \cite{GKV:hopf_torsion} is a special case of a general lifting theorem of a factorisation system on $\clA,\clX$ to $\algTotal\params\carriers$.
\end{itemize}

%% file: sections/2-foa.tex
\section{Fibrations of algebras}\label{sec_foa}
We assume the reader is familiar with the basic language of fibred categories, a standard reference for which is \cite{streicher2023fibred}. In particular, recall that the \emph{Grothendieck construction} \cite[Chapter 10]{Johnson2021pn} builds a correspondence between $\Cat$-valued pseudofunctors and (op)fibrations. This can be stated, for any category $\params$, as an equivalence of 2-categories, in either of the two forms
\begin{equation}\label{eq_grothCont}
	\Fib(\params) \simeq \Psd(\params^\op,\Cat)
	\qquad \text{ or } \qquad
	\opFib(\params) \simeq \Psd(\params,\Cat) \,,
\end{equation}
where $\mathsf{(op)}\Fib(\params)$ denotes the 2-category of (op)fibrations with codomain $\params$, as defined in \cite[$\S$2]{streicher2023fibred}. Such equivalence can be resticted to split (op)fibrations and functors.

The main object of study of the present paper is going to be a \emph{parametrized (endo)functor} and a particular way to obtain a fibration (or an opfibration) out of it.
\begin{definition}[Parametrized endofunctor, parametrized monad]\label{param_endofu}
Let $\clA,\clX$ be two categories; a \emph{parametrized endofunctor} (with \emph{category of parameters} $\params$ and \emph{category of carriers} $\carriers$) is a functor
\[\vxy{
F : \params\times\carriers \ar[r] & \carriers.
}\]
Clearly, every $F_A=F(A,-)$ is an endofunctor of $\clX$: as such, a parametrized endofunctor determines its mate $F':\clA \to \en\clX$. Whenever this doesn't give rise to ambiguity, we refer to both as `the parametrized endofunctor $F$'.

Similarly, a \emph{parametrized monad} is a parametrized endofunctor $T$ such that each $T_A=T(A,\firstblank) : \carriers\too\carriers$ is a monad (in its curried form, $T' : \params\to\Mnd\carriers$), and each $f : A\to B$ induces a monad morphism $T(A,\firstblank)\To T(B,\firstblank)$ in the sense of \autoref{def_mor_monads}.
\end{definition}
\begin{notation}
From now on, we adopt the implicit convention that parametrized endofunctors will be denoted with letters like $F,G,\dots$ while those that are parametrized monads will be denoted $S,T,\dots$; this will often allow to `type' a parametrized functor correctly and speed up the reading.
\end{notation}
\begin{definition}[Morphisms of monads]\label{def_mor_monads}
	A morphism in $\Mnd\carriers$ is a monoid homomorphism, when monads are regarded as monoids internal to $\en\carriers$; more explicitly, if $(T,\eta^T,\mu^T)$ and $(S,\eta^S, \mu^S)$ are monads on the same category $\carriers$, a morphism of monads $\alpha : S\To T$ consists of a natural transformation between the underlying functors, such that the equations
	\begin{itemize}
\item $\eta^T_X = \alpha_X \circ \eta^S_X$;
\item $\alpha_X \circ \mu^S_X = \mu^T_X \circ (\alpha\star \alpha)_X$,
	\end{itemize}
	witnessing compatibility with the units and multiplication of the two monads,	are satisfied. Here, $\alpha \star\alpha : TT\To SS$ is the horizontal composition of $\alpha$ with itself.%, regarded as a morphism $TT\To SS$.
\end{definition}
Let us now comment on our main construction. We pick apart the case of monads: denote $\EM(T)$ the category of \textEM algebras of a monad $T$, then every parametrized monad $T: \params \to \Mnd{\carriers}$ gives rise to a functor
\[\vxy[@R=0mm]{
\params^\op \ar[r] & \Cat\\
\param \ar@{|->}[r] & \EM(T_\param).
	}
\]
Under the equivalence of \eqref{eq_grothCont}, to this functor is associated a fibration, described as follows. The total category (\ie the domain) of the fibration is the category where
\begin{itemize}
	\item
an \emph{object} is a triple $(\param,\carrier,\xi : T_\param\carrier\to\carrier)$, more concisely denoted $\alg AX\xi$,
where $\xi$ is an \textEM algebra of the monad $T_\param$ and
	\item
a \emph{morphism}
$\alg AX\xi\to \alg BY\theta$
is a pair $(f,g)$ of morphisms
$f : A \to B$ in $\params$
and
$g : X \to Y$ in $\carriers$,
such that the square
\begin{equation}\label{mor_of_EM}
\vxy{
T_A X \ar[d]_{\xi}\ar[r]^{T_f g} & T_{B} Y  \ar[d]^{\theta} \\
X \ar[r]_g & Y
}
\end{equation}
in $\carriers$ commutes. Composition and identities are pairwise. 
\end{itemize}
% Composition and identities are induced by the ones in $\carriers$, using naturality of the transformation $T_f : T_A\To T_{B}$ induced by $f$: the composition
% \[\alg AX\xi\xto{(f,g)}\alg {B}{Y}{\theta}\xto{(f',g')}\alg{C}{Z}{\zeta}\]
% is given by
% \[\vxy[@R=1mm]{
%& T_AY \ar@{->}[rd]^{T_fY} &  & T_{B}Z \ar@{->}[rd]^{T_{f'}Z} &  \\
%T_AX \ar@{->}[rd]_{T_fX} \ar@{->}[ru]^{T_Ag} \ar@{->}[dddd]_{\xi} &  & T_{B}Y \ar@{->}[rd]_{T_{f'}Y} \ar@{->}[ru]^{T_{B}g'} \ar@{->}[dddd]^{\theta} &  & T_{C}Z \ar@{->}[dddd]^{\zeta} \\
%& T_{B}X \ar@{->}[ru]_{T_{B}g} &  & T_{C}Y \ar@{->}[ru]_{T_{C}g'} &  \\
%&  &  &  &  \\
%&  &  &  &  \\
%X \ar@{->}[rr]_{g} &  & Y \ar@{->}[rr]_{g'} &  & Z
%}\]
\begin{notation}
We denote the category so defined as $\algTotal[T]{\params}{\carriers}$.
\end{notation}
%\begin{remark}
%	A \emph{morphism of monads} is, for us, a natural transformation satisfying the conditions to be a monoid homomorphism, once monads are regarded as monoids in $\en\carriers$.
%The axioms in \autoref{def_mor_monads} make all reindexings $\alpha^*$ well-defined functors as the composition $\xi\circ\alpha_X$ is still an \textEM algebra.
%\end{remark}

\begin{definition}[\textEM fibration]\label{em_foa}
	Let $T : \params\times\carriers\too\carriers$ be a parametrized monad. The \defn{\textEM fibration} of $T$ (or the \emph{fibration of \textEM algebras} of $T$) is the functor
	\[
p^T  : \algTotal[T]{\params}{\carriers} \too \params \,, \quad
\alg{\param}{\carrier}{\xi} \mapsto \param \,.
	\]
\end{definition}
As for the fibrational part of the structure, given a morphism $f: \param' \to \param$ in $\params$, the functor $f^\ast  : \EM(T_\param) \to \EM(T_{\param'})$ computing the reindexing along $f$ is given by the mapping
\[
	\alg AX\xi \mapsto \alg {A'}X{\xi\circ T_fX}
\]
where the map $\xi : T_AX\to X$ is `reindexed' to $T_{A'}\carrier \xrightarrow{T_f X} T_\param \carrier \xrightarrow{\xi} \carrier$.

Every natural transformation $T_f : T_{A'}\To T_A$ is in fact a morphism of monads in the sense of \autoref{def_mor_monads}, so $\alg {A'}X{\xi\circ T_fX}$ is really an \textEM algebra for $T_{A'}$:
\begin{itemize}
	\item the compatibility with the unit is witnessed by the commutativity of the diagram
\[\vxy{
& X \ar@{->}[ld]_{\eta^{A'}_X} \ar@{->}[d]^{\eta_X^A} \ar@{=}[rd] &  \\
T_{A'}X \ar@{->}[r]_{T_fX} & T_AX \ar@{->}[r]_{\xi} & X
}\]
	\item the compatibility with the multiplication is witnessed by the commutativity of
\[\vxy{
T_{A'}T_{A'}X \ar@{->}[r]^{T_{A'}T_fX} \ar@{->}[dd]_{\mu^{A'}_X} \ar@{->}[rd]_{(T_f\star T_f)_X} & T_{A'}T_AX \ar@{->}[r]^{T_{A'}\xi} \ar@{->}[d]^{T_fT_AX} & T_{A'}X \ar@{->}[d]^{T_fX} \\
& T_AT_AX \ar@{->}[r] \ar@{->}[d]^{\mu^A_X} & T_AX \ar@{->}[d]^{\xi} \\
T_{A'}X \ar@{->}[r]_{T_fX} & T_AX \ar@{->}[r]_{\xi} & X
}\]
\end{itemize}
\begin{remark}\label{more_precisely}
	The action of the reindexing functors $f^* : \EM(T_A)\too \EM(T_{A'})$ introduced in \autoref{em_foa} can be described as carrying an $A$-algebra $\xi : T_AX\to X$ to the $A'$-algebra $f^*(\carrier,\xi)=\xi\circ T_f \carrier: T_{\param'} \carrier \to \carrier$. The commutativity condition of \eqref{mor_of_EM} expresses the fact that $f : X\to X'$ is a morphism of $T_{A'}$-algebras \emph{into the reindexing}, $\alg{\param}{\carrier}{\xi} \to f^*(\alg{\param'}{\carrier'}{\xi'})$.
\end{remark}
\begin{remark}
	An explicit description of the fibration of algebras of a parametrized endofunctor (or of a parametrized pointed endofunctor) is the same as above, except for the appropriate relaxation of the condition of $\xi$ being just a (pointed) endofunctor algebra.
\end{remark}
\begin{definition}[co\textEM opfibration]
	Dually, one can construct the opfibration
	\[
p_S  : \coalgTotal[S]{\params}{\carriers} \too \params
	\]
	of a parametrized comonad $S: \params \to \coMnd{\carriers}$. This category has% objects and morphisms of $\coalgTotal[S]\params\carriers$ are respectively
	\begin{itemize}
\item objects triples $\coalg AX\xi=(\param,\carrier,\xi : \carrier\to S_\param \carrier)$ with $\xi$ a co\textEM coalgebra for $S_\param$,
\item morphisms $\coalg AX\xi\to \coalg BY\theta$ the pairs $(f,g)$ with $f : \param\to B$ in $\params$ and $g : \opFibReindex{f}(X,\xi)\to(Y,\theta)$ a $S_B$-coalgebra morphism.
	\end{itemize}
\end{definition}
\begin{notation}\label{implicit_endo_foa}
	We often write $\algTotal{\params}{\carriers}$ for the total categories, omitting the super- or subscript, if it is clear from the context which parametrized monad or comonad we are using. In the present note the same notation will also be used for (possibly pointed) endofunctor (co)algebras.\footnote{The only situation when this might cause confusion is when one considers the category of endofunctor algebras for the underlying endofunctor of a given monad $T$; we will never be in such a situation here.}%the objects of which are still pairs $\alg AX\xi$ where now $(X,\xi)$ is a mere \emph{endofunctor} algebra with respect to the parameter $A$, and morphisms are as in \eqref{mor_of_EM}.
\end{notation}

Evidently, there is a lot of freedom in how to vary the definition of the fibration of algebras. One can use endofunctor algebras, pointed endofunctor algebras, \textEM, co\textEM, etc. for a given endofunctor admitting more and more structure. So, if one wants to taxonomize the situation, and avoid confusion, they have to adopt a more systematic approach to notation. Given that the present work is aimed at presenting the benefits and basic features of fibrations of algebras, we leave the task of systematisation to a separate note.

% \autoref{gna} collects the categories we can define out of a parametrized endofunctor $F$, a pointed such endofunctor $P$, a parametrized monad $T$, a copointed parametrized endofunctor $Q$, a parametrized comonad $S$.

For the time being, it is convenient to record the following nomenclature.
%\greta{I see no result following this sentence, and I don't understand the sentence before it: what systematic approach are we referring to? (I think \textit{I} know, but it should be clearer.)}
\begin{definition}\label{univ_of_endo}
	The \emph{universal fibration $U_\Alg$ of endofunctor algebras} is the fibration associated to the pseudofunctor
	\[
\vxy{
\Alg : \en\carriers^\op\ar[r] & \Cat
}
	\]
	sending an endofunctor $F : \carriers\too\carriers$ to its category of endofunctor algebras and algebra homomorphisms.
\end{definition}
Actually, $\Alg$ is not only a pseudofunctor but a \emph{functor}, so that its associated fibration is split. As it will not make much of a difference in our exposition, we abstain from mentioning this at every turn, but the fibration-oriented reader should definitely keep this in mind.
\begin{definition}\label{univ_of_em}
	The \emph{universal fibration $U_\EM$ of \textEM algebras} is the fibration associated to the pseudofunctor
	\[
\vxy{
\EM : \Mnd{\carriers}^\op\ar[r] & \Cat
}
	\]
	sending a monad $T : \carriers\too\carriers$ to its category of \textEM algebras and algebra homomorphisms, and a morphism of monads to the reindexing functor $\alpha^* : \EM(T)\too\EM(S)$ that sends $(X,\xi)\mapsto (X, \xi\circ\alpha_X)$.
\end{definition}

Now, the following two observations follow unwinding the definition of a (strict) pullback in $\Cat$; alternatively, they can be seen as a straightforward particular instance of the Grothendieck construction.
\begin{remark}
	Let $F : \params \too\en\carriers$ be a parametrized functor; the fibration $\algTotal[F]{\params}{\carriers}$ of \autoref{implicit_endo_foa} arises pulling back $F$ along the universal fibration $U_\Alg$ of endofunctor algebras of \autoref{univ_of_endo}.
\[\vxy{
\algTotal[F]{\params}{\carriers} \ar[r]\ar[d]\xpb & \int\Alg\ar[d]^{U_\Alg}\\
\params \ar[r]_F & \en\carriers
}\]
\end{remark}
Observe that one can produce a pullback diagram presenting $\algTotal[F]{\params}{\carriers}$, meaning that all fibrations of algebras arise pulling back from a universal one. Such a fibration of algebras is obtained when $\clA=\en\clX$ is itself the category of endofunctors, acting with representation the identity functor $\clA\too\en\clX$. With that in mind, we will begin denoting the total category of the universal fibration of algebras (cf. \autoref{univ_of_endo}) as
\[
\int\Alg_\carriers = \algTotal{\en\carriers}{\carriers},
\]
and similarly for all its variants. In monoid theory (but, more frequently, in group theory), the semidirect product $\text{End}(M)\ltimes M$ (resp., $\text{Aut}(G)\ltimes G$) under the identity representation is called the \emph{holomorph} monoid (resp., group).
A step-by-step comparison of the case of monoids and that of categories is detailed in \autoref{subexam_algebras}, see in particular \eqref{pb_endo} and \eqref{pb_em} below.

Similarly,
\begin{remark}
	Let $T : \params\times\carriers\too\carriers$ be a parametrized functor; the fibration $\algTotal[T]{\params}{\carriers}$ of \autoref{em_foa} arises pulling back $F$ along the universal fibration of \textEM algebras of \autoref{univ_of_em}.

	Explicitly, there is a pullback diagram presenting $\algTotal[F]{\params}{\carriers}$, depicted in \eqref{pb_em} below.
\end{remark}
Considering \emph{free} algebras for a parametrized monad $T$ is possible but requires a bit of care:
\begin{definition}
	The \emph{universal opfibration of Kleisli categories} is the total category of the pseudofunctor $\Kl : \Mnd{\clX} \too \Cat$ sending a monad $(T, \eta, \mu)$ to its Kleisli category. $\Kl$ is defined on morphisms of monads $\alpha : T\To S$ as follows:
	\begin{itemize}
\item each $\opFibReindex{\alpha} : \Kl_\clX(T)\too\Kl_\clX(S)$ is the identity on objects functor
\item mapping a Kleisli morphism $f : X\to TY$ to $\alpha_Y\circ f : X\to TY\to SY$.
	\end{itemize}
	Functoriality of the assignment $\alpha\mapsto \opFibReindex{\alpha}$ follows from the fact that $\alpha$ is a monoid homomorphism.
\end{definition}
\begin{definition}\label{def_kl_opfib}
	Let $T : \params \too \Mnd{\carriers}$ be a parametrized monad. The \emph{opfibration of free algebras} on $T$ is
	the category having
	\begin{itemize}
\item as objects the pairs $(A,X)$ where $A\in\params$ is an object and $X\in\Kl(T_A)$;
\item as morphisms $(u,f):(A,X)\to (A',Y)$ the pairs $u : A\to A'$ and $f : \opFibReindex{u} X \to Y$ (a Kleisli morphism in $\Kl(T_{A'})$).
	\end{itemize}
	The functor $\opFibReindex{u} = T_{u,\ast} : \Kl(T_A)\too \Kl(T_{A'})$ acts as follows:
	\begin{itemize}
\item the identity on objects;
\item on morphisms a Kleisli map $f : X\to T_AY$ goes to $T_uY\circ f : X\to T_AY\to T_{A'}Y$.
	\end{itemize}
\end{definition}
Similar constructions define the universal fibration of free co\textEM coalgebras and the fibration of free coalgebras of a parametrized comonad: since one of our main examples, \autoref{fib_reg_rep}, is a fibration of cofree colgebras, we spell it out in full.
\begin{definition}\label{def_cokl_fib}
	Let $S : \params \too \coMnd{\carriers}$ be a parametrized comonad. The \emph{fibration of cofree coalgebras} modeled on $S$ is
	the category having
	\begin{itemize}
\item as objects the pairs $(A,X)$ where $A\in\params$ is an object and $X\in\coKl(S_A)$;
\item as morphisms $(u,f):(A,X)\to (A',Y)$ the pairs $u : A\to A'$ and $f : X \to  \opFibReindex{u}Y$ (a coKleisli morphism in $\coKl(T_A)$).
	\end{itemize}
	The functor $\opFibReindex{u} = S_{u,*} : \coKl(S_{A'})\too \coKl(S_A)$ acts as follows:
	\begin{itemize}
\item the identity on objects;
\item on morphisms a coKleisli map $f : S_{A'}X\to Y$ goes to $f\circ S_uX : S_AX\to S_{A'}X\to Y$.
	\end{itemize}
\end{definition}
\begin{remark}
It would be tempting now to assert that given a monad $T$ on $\carriers$ the comparison functors $K_T : \Kl_\carriers(T)\to\EM_\carriers(T)$ are the components of a transformation $\Kl_\carriers\To\EM_\carriers$ of sorts, and maybe even of a morphism of fibrations
\[
\vcenter{\xymatrix{
    \int\Kl_\carriers \ar[rr]\ar[dr] && \int\EM_\carriers \ar[dl] \\
    & \Mnd\carriers}}
\] between the Kleisli fibration and the Eilenberg-Moore fibration.

A moment of reflection is however enough to realize that the statement doesn't even typecheck, given the opposite variance of the functors. Even resorting to the notion of a (pseudo)\emph{dinatural} family of maps the diagram
\[\vcenter{\xymatrix{
    \Kl_\carriers(T)\ar[r]^K\ar[d]_{\alpha_*} & \EM_\carriers(T) \\
    \Kl_\carriers(S)\ar[r]_{K'} & \EM_\carriers(S) \ar[u]_{\alpha^*}
}}\]
is not commutative; the best one can do is prove that it is \emph{laxly} commutative, i.e. that there is a 2-cell $\kappa : K\To \alpha^* \circ K' \circ\alpha_*$ closing the square.\footnote{We omit the proof of this statement that essentially relies on the fact that $\alpha^*$ acts as the identity on carriers and just changes the algebra maps, and moreover it acts trivially on algebra morphisms as well; the components of $\kappa$ are determined as the components of $\alpha$, $\alpha_A : (TA,\mu_A^T)\to (SA,\mu_A^S\circ\alpha_{SA})$; the naturality of $\kappa$ at a Kleisli map $\varphi : A\to TB$ boils down to the commutativity of \[\notag\vcenter{\xymatrix{
    TA \ar[rrr]^{\mu^T_B\circ T\varphi}\ar[d]_{\alpha_A}&&& TB\ar[d]^{\alpha_B} \\
    SA \ar[r]_{S\varphi} & STB \ar[r]_{S\alpha_B} & SSB \ar[r]_{\mu^S_B} & SB
}}\] which can be broken down to the fact that $\alpha$ is a monad morphism.}
\end{remark}
% \newcommand{\hcyan}{\cellcolor{teal!20}}
% \newcommand{\horan}{\cellcolor{orange!20}}
% \begin{table}[]
%\centering
%\begin{tabular}{@{}l|l|l@{}}
%\toprule
%& fibrations                         & opfibrations                       \\ \midrule
%& \hcyan $\barealg[F]\clA\clX$          &                                    \\
%algebras   & \hcyan $\barealg[P]\clA\clX[*]$ & \horan $\klalg\clA\clX$            \\
%& \hcyan $\barealg[T]\clA\clX[\EM]$            &                                    \\\midrule
%&                                    & \horan $\barecoalg\clA\clX$        \\
%coalgebras & \hcyan $\barecoalg[S]\clA\clX$          & \horan $\barecoalg[\circ]\clA\clX$ \\
%&                                    & \horan $\coemalg\clA\clX$          \\ \bottomrule
%\end{tabular}
%\caption{A table containing all
%\label{gna}
%total categories we can define from parametrized functors, organized by type (fibration or opfibration,
%depending on the variance of the functors in study). These categories depend on a parametrized endofunctor $F$, a pointed such endofunctor $P$, a parametrized monad $T$, a copointed parametrized endofunctor $Q$, a parametrized comonad $S$.}
% \end{table}
\subsection{Matters of functoriality}\label{matters_of_fun}
This subsection aims to prove the following theorem and to draw some of its consequences.
\begin{theorem}\label{yoink_a_2_fibration}
	There exists a 2-category denoted $\Cat\ltimes^\ell\Cat$, so that sending $(\params,\carriers,F)$ to $\algTotal[F]{\params}{\carriers}$ is a 2-functor $\Cat\ltimes^\ell\Cat\too\Cat$.
\end{theorem}
We begin by building $\Cat\ltimes^\ell\Cat$.
\begin{definition}[The total category $\Cat\ltimes^\ell\Cat$]\label{Cat_ell_cat}
	Define the 2-category $\Cat\ltimes^\ell\Cat$ having
	\begin{itemize}
\item as objects the triples $(\params,\carriers,F)$ where $F : \params\times\carriers\too\carriers$ is a parametrized endofunctor;
\item 1-cells are \emph{oplax morphisms of algebras} \ie pairs $(U,V,\delta) : (\params, \clX,F)\to (\clB,\clY,G)$ where $U : \params\too\clB$, $V : \clX\too \clY$ are functors and $\delta$ is a 2-cell filling the diagram below;
\[\label{lax_algebra_mor}\vxy{
\clA\times\clX \ar[r]^-F \ar[d]_{U\times V}& \clX \ar[d]^U \\
\clB\times\clY \ar[r]_-G & \clY \ultwocell<\omit>{\delta}
}\]
\item 2-cells $(U,V,\delta)\To (U',V',\delta')$ are pairs $\omega : U\To U',\nu : V\To V'$ of natural transformations such that the following equality of pasting 2-cells holds.
% $(U,\upsilon)\To (V,\omega)$ are natural transformations $\kappa : V\To U$ (note the reversal) such that the equality of pasting
\[\label{2cell_eqn}\vxy[@R=12mm@C=12mm]{
\clA\times\clX \ar[r]^-F \ar@/^1pc/[d]\ar@/_1pc/[d] & \clX \ar@{}[dr]|{=}\ar[d]^V & \clA\times\clX \ar[r]^-F \ar[d]& \clX \ar@/^.66pc/[d]\ar@/_.66pc/[d] \\
\clA\times\clY\utwocell<\omit>{\omega\times\nu} \ar[r]_-G & \clY \ultwocell<\omit>{\delta}& \clA\times\clY \ar[r]_-G & \clY\ultwocell<\omit>{\delta'}\utwocell<\omit>{\nu}
}\]
	\end{itemize}
\end{definition}
\begin{remark}[On the notation $\Cat\ltimes^\ell\Cat$]
	We choose to denote $\Cat\ltimes^\ell\Cat$ the category in \autoref{Cat_ell_cat}, since one can see the domain of the fibration of algebras construction as arising from a parametrized $2$-endofunctor
	\[\Cat \too \en \Cat \,, \quad \params \mapsto \params \times \firstblank\]
	of which we consider \emph{oplax} endofunctor algebras, in the same sense as \cite{blackwell1989two}. This hints at the presence of a 2-dimensional, laxified version of the theory of fibration of algebras, for parametrized endo-2-functors of 2-categories and bicategories, where we consider 2-fibrations of various kinds of algebras attached to $F$; a thorough study of the matter would lead us astray.
\end{remark}
Let's instead focus on specializing \autoref{Cat_ell_cat} to the case of parametrized monads (we will need this in the proof of \autoref{theorem:_the_fibration-monad_adjunction}) and on proving \autoref{yoink_a_2_fibration}.
\begin{remark}\label{same_but_4monads}
  One can similarly define a category $\Cat\ltimes^{\ell,\EM} \Cat$ where 
  \begin{itemize}
    \item objects are triples $(\params,\carriers,F)$ where $T \colon \params \too \Mnd{\carriers}$ is parametrized monad;
    \item 1-cells are triples, $(U,V,\delta)$ as in \autoref{Cat_ell_cat}, where in addition $\delta$ satisfies the equations 
      \begin{gather}
          \xymatrix@C=1.2cm{
          \clA\times\clX \ar[d]_{U\times V} \ar[r]^T & \clX \ar[d]^V\ar@{}[drr]|{\text{\large =}} && \clA\times\clX \ar[d]_{U\times V} \ar[r]^\pi \drtwocell<\omit>{={}}& \clX \ar[d]^V\\
          \clB\times\clY & \clY \ultwocell<\omit>{<1>\delta}\ltwocell~'{\dir{}}~`{\dir{<}}^{\pi}_{S}{\eta^S} && \clB\times\clY \ar[r]_\pi & \clY
          }\notag\\ 
          \footnotesize\xymatrix@C=4mm{
          \clA\times\clX \ar[r]\ar[d]_{U\times V} \ar@/^2pc/[rrr]^-T  & \clA\times\clA\times\clX \ar[r]\ar[d] & \clA\times\clX \ltwocell<\omit>{<2>\mu^T}\ar[r]\ar[d] & \clX\ar[d]^V \ar@{}[drr]|{\text{\large =}} && \clB\times\clY \ar[d]_{U\times V}\ar@/^2pc/[rrr]^-T & && \clY\ar[d]^V\\
          \clB\times\clY\ar[r] & \clB\times\clB\times\clY\ar[r]\ultwocell<\omit>{={}} & \clB\times\clY\ar[r] \ultwocell<\omit>{\hat\delta}& \clY \ultwocell<\omit>{\delta}&& \clB\times\clY\ar[r]\ar@/^2pc/[rrr] & \clB\times\clB\times\clY\ar[r] & \ltwocell<\omit>{<2>\mu^S}\clB\times\clY\ar[r] & \clY\ulltwocell<\omit>{<3>\delta}\\
          }\label{par_mnd_axioms}
      \end{gather}
      making it a morphism of monads;\footnote{The map $\hat\delta$ is defined as the whiskering $(U\times\clY) * (\clA\times\delta)$ and altogether the lower pasting diagram where it appears consists of the horizontal composition of $\delta$ with itself, regarded as a 2-cell between coKleisli maps: this is the basis on which we build intuition for \autoref{as_monads_in_cokl}, proving that in fact a parametrized monad $T$ as above is a monad in $\coKl(\clA\times\firstblank)$.}
    \item 2-cells are pairs $(\omega,\nu)$ as in \autoref{Cat_ell_cat}.\eqref{2cell_eqn}.%, where in addition the commutativities 
      %\[bla\]
      %hold, making $(\omega,\nu)$ a 2-cell between monads.
  \end{itemize}
\end{remark}
\begin{proof}[Proof of \ref{yoink_a_2_fibration}]
	\autoref{Cat_ell_cat} entails that given a 1-cell $(U,V,\delta) : (\clA,\clX,F)\to(\clB,\clY,G)$, the assignment
	\[
\vxy{
(X,\xi)^A \ar@{|->}[r]& (VX, V\xi\circ \delta_{AX})^{UA}
}
	\]
	is a functor $U\ltimes(V,\delta) : \algTotal[F]{\clA}{\clX}\too\algTotal[G]{\clB}{\clY}$. On morphisms of algebras $(f,g) : (X,\xi)^A\to(X',\xi')^{A'}$ this functor is defined using the action of $V$ on morphisms, and by the equation in \eqref{lax_algebra_mor} the diagram
	\[\vxy{
G_{UA}VX\ar[r]^-{G_{Uf}Vg}\ar[d]_{\delta_{AX}} & G_{UA'}VX' \ar[d]^{\delta_{A'X'}}\\
VF_AX \ar[d]_{V\xi}\ar[r]^-{VF_fg}& VF_{A'}X' \ar[d]^{V\xi'}\\
VX \ar[r]_{Vg} & VX'
}
	\] commutes, so that $U\ltimes(V,\delta)(f,g)=(Uf,Vg)$ defines a morphism $(VX, V\xi\circ \delta_{AX})^{UA}\to (VX', V\xi'\circ \delta_{A'X'})^{UA'}$ in the fibration of algebras $\algTotal[G]\clB\clY$.

	At the level of 2-cells, axiom \eqref{2cell_eqn} ensures that a pair $(\omega,\nu)$ defines an algebra map
	\[\vxy[@C=2cm]{
(VX,V\xi\circ\delta_{AX})^{UA}\ar[r]^-{(\omega_A,\nu_X)} & (V'X, V'\xi\circ\delta_{AX}')^{U'A};
}\]
	in fact \eqref{2cell_eqn} is exactly the commutativity condition of the square below.
	\[\label{vnaodnoa}\vxy[@C=2cm]{
G_{UA}VX \ar[r]^-{G(\omega_A,\nu_X)}\ar[d]_{\delta_{AX}}& G_{U'A}V'X\ar[d]^{\delta'_{AX}}\\
VF_AX \ar[r]_-{\nu_{F_AX}}& V'F_AX
}\]
	Now, the family of maps $(\omega_A,\nu_X)$ is natural because the square of algebra morphisms in $\algTotal[G]\clB\clY$
	\[\vxy{
(VX,V\xi\circ\delta_{AX})^{UA} \ar[r]^-{(Uf,Vg)}\ar[d]_{(\omega_A,\nu_X)} & (VX',V\xi'\circ \delta'_{AX})^{UA'}\ar[d]^{(\omega_{A'},\nu_{X'})}\\
(V'X,V'\xi\circ\delta_{AX})^{U'A} \ar[r]_-{(U'f,V'g)} & (V'X',V'\xi'\circ\delta'_{A'X'})^{U'A'}
}\]
	commutes thanks to the naturality of $\nu$ and to \eqref{vnaodnoa}, together with the rule defining the composition law in $\algTotal[G]\clB\clY$.
\end{proof}
\subsection{A whirlwind tour of examples of fibrations of algebras}\label{whirl}
% \fosco{More examples now that we have space}
The scope of this section is to collect examples of diverse origin. 

Among the simplest instances of parametrized monads and comonads there are the ones arising from a Cartesian or coCartesian category of parameters acting on itself by co/product. This quite specific setting is already enough to provide interesting examples. Subsequently we show how more general monoidal categories offer equally compelling examples, and we start outlining the universal property enjoyed by the construction in \autoref{subexam_algebras}.

More examples will follow when we will have established additional bits of the theory (e.g. how colimits are computed in total categories of the form $\algTotal\clA\clX$), see \autoref{subsec:colims}, and the 2-dimensional universal property thereof, \autoref{as_monads_in_cokl}.
\begin{example}[Simple fibration]\label{fib_reg_rep}
	In this example the categories of parameters and carriers coincide, \ie $\params=\carriers$; moreover, $\carriers$ is a Cartesian category.

	If we consider the (transpose of the) Cartesian product functor $P:\carriers \too \en \carriers$, sending $\param$ to the endofunctor $P\param = \param \times \firstblank$, then each functor $P\param$ carries the structure of a comonad, sometimes called the \emph{coreader comonad}, which collectively make $P$ into a parametrized comonad. The coKleisli fibration $\var{\bfs(\carriers)}{\carriers}$ of $P$ goes under the name of \emph{simple fibration} \cite[Definition 1.3.1]{CLTT}. Meanwhile, the opfibration of co\textEM coalgebras of $P$ is the codomain opfibration $\var[\text{cod}]{\carriers^\rightarrow}{\carriers}$.

	Dually, the parametrized monad $\param \mapsto \param + \firstblank$, given by the coproduct and known as the \emph{exception monad}~\cite{MOGGI199155}, will respectively yield the \emph{simple opfibration} $\var{\bfs\bfo(\carriers)}{\carriers}$ of \cite[p. 511]{CLTT} and the domain fibration as the Kleisli and \textEM fibrations.
\end{example}
The typical fibre $\bfs(\carriers)_\param$ of the simple fibration is a category that we denote $\carriers\slice \param$; it has the same objects as $\clX$, and as hom-sets $\bfs(\carriers)_\param(X,Y)\coloneq \carriers(\param \times X,Y)$.
It is called the \emph{simple slice} category over $\param$, equipped with the structure coming from coKleisli composition.
\begin{remark}
	The simple fibration is important in the semantics of programming languages, predominantly in the modelling of contexts of free variables (in the elimination rules of various type formers, such as inductive types). It is also important for characterising strong monads, that are used to model computational effects---in fact, a strong monad is a fibred monad on the simple fibration~\cite{MOGGI199155}.
\end{remark}
\begin{example}[Semiautomata]\label{example:_semiautomata}
	For every object $A$ of a monoidal category $(\params,\otimes)$, we can consider the endofunctor $\param \otimes \firstblank$. The {endofunctor algebras} for this are morphisms $d : A\otimes X\to X$, which are known as \emph{(monoidal) semiautomata} or \emph{Medvedev automata} (cf. \cite{Kilp2000,Ehrig}). Collectively, these will form the \emph{fibration of semiautomata}.
\end{example}
\begin{remark}
	In the case of a Cartesian monoidal category $(\params,\times)$, note that this is quite distinct from the coalgebras for the canonical comonad structure that functors $A \times \firstblank$ can be equipped with. As simple as it is, this constitutes an instance of the additional freedom allowed by a unified approach to the problem of studying fibrations arising from parametrized endofunctors: given the regular Cartesian representation $PA=A\times\firstblank$ of $\carriers$ on itself, one can consider
	\begin{itemize}
\item cofree co\textEM $P$-coalgebras, or alternatively
\item all co\textEM $P$-coalgebras, thus falling into the two cases of \autoref{fib_reg_rep},
\item endofunctor $P$-algebras, thus falling into \autoref{example:_semiautomata}.
	\end{itemize}
	Something similar happens anytime the same parametrized endofunctor carries more than one structure, and this might pose the problem of a slight abuse of notation or ambiguity: context will always solve any potential issues.
\end{remark}
The following examples can be seen as enriching the structure of a Medvedev automaton with an \emph{output} morphism $s : A\otimes X\to B$ yielding the result of the computation performed by $d : A\otimes X\to X$. A reference for the coalgebraic description of such a task can be found in \cite[Ex. 2.3.2]{Jacobs2016}, while the (straightforward) description using pullbacks comes from \cite[3.5]{noi:completeness}.
\begin{example}[Categories of automata]\label{example_mealy}
	Let $\clK$ be a monoidal closed category. For every pair of objects $A,B\in\clK$ one defines the category $\Mly(A,B)$ of \emph{Mealy automata} with input $A$ and output $B$ as the strict 2-pullback of forgetful functors
	\[\vxy{
\Mly(A,B) \ar[r]\ar[d]\xpb & \clK/[A,B]\ar[d]\\
\coAlg([A,\firstblank]) \ar[r] & \clK
}\]
	where $\coAlg([A,\firstblank])$ is the category of coalgebras for the functor $X\mapsto[A,X]$ and $\clK/[A,B]$ the slice category over $[A,B]$. As such, $\Mly(A,B)$ is the category of $R_{AB}$-coalgebras, where $R_{AB}X:=[A,X\times B]$; from this, \ie from the parametric endofunctor $R : (\clK^\op\times\clK)\times \clK\too\clK : (A,B,X)\mapsto [A,X\times B]$, we obtain the \emph{total Mealy opfibration} (of endofunctor coalgebras) $\Mly$ over $\clK^\op\times\clK$.

	Similarly, the category $\Mre(A,B)$ of \emph{Moore automata} with input $A$ and output $B$ is the strict 2-pullback of forgetful functors
	\[\vxy{
\Mly(A,B) \ar[r]\ar[d]\xpb & \clK/B\ar[d]\\
\coAlg([A,\firstblank]) \ar[r] & \clK
}\]
	from which we obtain the \emph{total Moore fibration} (of endofunctor coalgebras) $\Mre$ over ${\clK^\op\times\clK}$.
\end{example}
\begin{example}[Modules] \label{example_2}
	To get a monad version of \autoref{example:_semiautomata}, note that every monoid structure on $A$ equips the endofunctor $A\otimes \firstblank$ with the structure of a monad, giving us a parametrized monad
	(sometimes called the \emph{writer monad}~\cite{MOGGI199155})
	\[
\Mon(\params,\otimes) \too \Mnd{\params}\,, \quad A \mapsto A \otimes \firstblank \,.
	\]
	The resulting \textEM fibration is called the \emph{fibration of modules} \cite{Quillone1970,Frankland}.
\end{example}

For the monoidal category $(\Ab,\otimes)$, Example~\ref{example_2} is the \textEM fibration of the parametrized monad
\[
	\Ring = \Mon(\Ab,\otimes) \too \Mnd{\Ab}\,, \quad A \mapsto A \otimes \firstblank \,,
\]
which is a fibration known to Quillen (cf. \cite{Weibel1994} for a discussion of its structure, mostly its co/completeness, and \cite{Fran1} for some applications to Quillen cohomology).

We now want to make the intuition that these parametrized monads behave {\it as} group actions a bit more precise. In particular, recall that for a group $G$ one can form the semidirect product $\text{Aut}(G)\ltimes G$ when $\text{Aut}(G)$ acts on $G$ by evaluating an automorphism $f$ on an element $g$: this is called the \emph{holomorph} of $G$. The same construction carries over to monoids, as $\End(M)$ similarly acts on $M$ via the evaluation map. We can recapture this example, and generalise it to categories, via a suitable (universal) fibration of algebras.

%\ulo{Use $\id$ for identity everywhere}
\begin{example}[Classifying fibrations of algebras, again]\label{subexam_algebras}
	When $\en\carriers$ is the monoidal category of endofunctors of a fixed category $\carriers$, the parametrized endofunctor
	\[
\id  : \en\carriers \too \en\carriers \,, \quad F \mapsto F(\firstblank)
	\]
	we recover the category of \autoref{univ_of_endo}. As we have been hinting at, this provides a sort of universal paradigm for our construction, in the form of an analogue of the notion of {holomorph} in group theory. The results of \autoref{univ_of_endo}, \autoref{univ_of_em} can be rephrased as follows:	for any other parametrized endofunctor $F: \params \too \en \carriers$, its fibration of algebras $p^F$ can be constructed by taking the following pullback in $\Cat$.
	\[\label{pb_endo}
\vxy{
{\algTotal[F]{\clA}{\carriers}}\ar[r]\ar[d]_{p^F}\xpb & \algTotal{\en \carriers}{\carriers}  \ar[d]^U\\
\clA \ar[r]_F & \en\carriers
}
	\]
\end{example}
Reasoning similarly, we get that every \textEM fibration $p^T$ can be constructed by pulling back the classifying \textEM fibration along the parametrized monad
\[\label{pb_em}
	\vxy{
{\algTotal[T]{\clA}{\carriers}}\ar[r]\ar[d]_{p^T}\xpb & \algTotal{\Mnd{\carriers}}{\carriers}  \ar[d]^U\\
\clA \ar[r]_T & \Mnd{\carriers}
	}
\]
where the functor $U$ is obtained applying the Grothendieck construction to the pseudofunctor of \autoref{univ_of_em}.

\begin{example}[The fibration of cocommutative Hopf algebras]\label{cartier_roba}
	Consider an al\-ge\-bra\-i\-cal\-ly closed field $k$ of characteristic zero; an important theorem of Cartier, Gabriel and Kostant \cite{Kostant1977}, \cite[5.10.2]{EGNO15} asserts that every cocommutative Hopf algebra $H$ over $k$ arises from a semidirect product of a group $G$ (or rather, its group algebra $k[G]$) acting on a Lie algebra $L$ over $k$. More formally, the correspondence $G\mapsto k[G]\emdash\Lie$, sending a group to the category of Lie algebras with an action of $k[G]$, is a (contravariant) functor, defining, in a similar fashion as \autoref{example_2} above, a fibred category collecting the \textEM algebras of the monads $k[G]\otimes\firstblank$.
	If we denote the total category of such fibration as $\emalg{\Grp}{\Lie}$ and write the fibration associated to $G\mapsto k[G]\emdash\Lie$ as a `first projection' map
	\[\label{cchopf}
\vxy{
\emalg{\Grp}{\Lie} \ar[r]^-p & \Grp
}
	\]
	we can reformulate the CGK theorem as a fibred equivalence of categories: call $(\firstblank)^\text{g} : \CCHopf\too\Grp$ the functor sending a cocommutative Hopf algebra $H$ to its group $H^\text{g}$ of group-like elements, then CGK states that $(\firstblank)^\text{g}\cong p$ in the category $\Fib(\Grp)$ of fibrations over $\Grp$.
\end{example}

\begin{example}[Polynomials in the sense of \protect{\cite{gambino_kock_2013}}]\label{gambo_kock_pol}
	Given a locally Cartesian closed ca\-te\-go\-ry $\clE$, we can define a category $\Poly{I}$ of polynomials having as objects the diagrams $\fkf : I \xot s B \xto f A\xto t I$ in $\clE$, and as morphisms $(v,u):\fkf=(A,B;s,f,t) \to (A',B'; s',f',t')=\fkf'$ the pairs $u : A\to A'$, $v : B\to B'$ such that
	\begin{itemize}
		\item $s'\circ u=s$;
		\item $v\circ f = f'\circ u$ and the resulting square is a pullback;
		\item $t'\circ v=t$.
	\end{itemize}
	To each object $\fkf : I \xot{s} B \xto{f} A \xto{t} I$ of $\Poly{I}$ one can associate a polynomial endofunctor $P_\fkf$ on $\clE/I$ defined using the parametrized adjunction $\Sigma_g \dashv \Delta_g \dashv \Pi_g$ (with parameter a morphism $g : U\to V$) as the composite $P_\fkf :\clE/I \xto{\Delta_s} \clE/B \xto{\Pi_f} \clE/A \xto{\Sigma_t} \clE/I$
	to the effect that a morphism $(v,u):\fkf\to\fkf'$ induces a natural transformation $P_\fkf \To P_{\fkf'}$. We obtain a fibration having typical fibre $\Alg_\clE(P_\fkf)$.
	%
	%	This example captures a great deal of interesting problems and results in one fell swoop: for example the treatment of (co)induction as in \cite{10.1006/inco.1998.2725} and recent work on polynomial pseudomonads for dependent type theory \cite{e884ea388bde4673b133e0b61db5b158}.
\end{example}

%% file: sections/3-foa.tex
\section{A characterization, and on semidirect products}
In this section we focus on better understanding inherent features of the fibration of algebras construction: first we provide a new characterization by means of a certain \textEM object (cf. \autoref{subsec:parametricity}), and then we use it to further explore how our construction is related to the semidirect product intuition, and to express its relation to \emph{graded monads} (cf. \autoref{subsec:act}).

\subsection{Parametricity via the product}\label{subsec:parametricity}
While it is natural to define fibrations of algebras through the application of the Grothendieck construction to a pseudofunctor, at times one might need alternative approaches; some are quite more elegant and allow for equivalent characterizations of fibrations of algebras.

Every parametrized functor ${F  : \params \too \en{\carriers}}$ induces an endofunctor
\[
	\hat{F}  \coloneqq (\pi_\params,F)  : \params \times \carriers \too \params \times \carriers \,, \quad (\param,\carrier) \mapsto (\param,F_\param\carrier)
\]
and the endofunctors $\hat{F}$ that are of this form are precisely the ones fibered over the projection in the sense of making the triangle
\[\vxy{
	\params \times \carriers
	\ar[rr]^{\hat{F}}
	\ar[dr]_{\pi_\params}
	& & \params \times \carriers
	\ar[dl]^{\pi_\params} \\
	& \params &
	}\]
commute.\footnote{A technical point: for practical reasons, we only ask the triangles in the slice-2-category to commute up to isomorphism.} As trivial as this observation might seem, it hides an elegant characterization: $\pi_\params$ is the trivial bundle over $\params$, \ie a specifically simple object
of the slice $2$-category $\Cat/\params$.

From this point of view, parametrized monads $T  : \params \too \Mnd{\carriers}$ correspond to monads
in the slice-2-category $\Cat/\params$, on the projection functor $\pi_\params  : \params \times \carriers \to \params$ (an \emph{object} of $\Cat/\params$). Thus, fibrations of algebras
form but a subclass of the formal theory of monads in the sense of \cite{Street1972}. This enables us to make immediate use of the theory and results that come for free when one studies 2-categories of monads, and their algebras.
\begin{theorem} \label{theorem:_product_presentation}
	The \textEM fibration $p^T  : \algTotal{\params}{\carriers} \fib \params$ of a parametrized monad $T  : \params \too \Mnd{\carriers}$ is equivalent, in the slice-2-category $\Cat/\params$, to
	\[
		\EM(\hat{T}) \xrightarrow{U^{\hat{T}}}\params \times \carriers \xrightarrow{\pi_\params} \params \,,
	\]
	where $U^{\hat{T}}$ is the forgetful functor from the category of  \textEM algebras of the monad
	\[
		\hat{T}  : \params \times \carriers  \too \params \times \carriers \,, \quad (\param,\carrier) \mapsto (\param, T_\param\carrier) .
	\]
	The functor $\var[p^T]{\algTotal{\params}{\carriers}}{\params}$ is also the (internal) \textEM object of the monad $\hat{T}$ in the slice-2-category $\Cat/\params$, on the projection $\var[\pi_\params]{{\params \times \carriers}}{\params}$.
\end{theorem}
\begin{proof}
	Clearly algebras for $\EM(\hat{T})$ are of the kind $(\id[\param],\xi) : (\param,T_\param \carrier)\to(\param,\carrier)$ with $\xi$ a $T_\param$ algebra, and similarly for their morphisms. The obvious correspondence
	\[
		(\param,\carrier,\xi : T_\param \carrier\to\carrier)\leftrightarrows(\id[\param],\xi : T_\param \carrier\to\carrier)
	\]
	makes an equivalence in $\Cat/\params$.
	As for the second part, we show its universal property. First of all consider  in $\Cat/\params$ the following 2-cell $\beta : \hat{T} U^{\hat{T}} \Rightarrow U^{\hat{T}}$,
	\[\vxy{
		\algTotal{\params}{\carriers}
		\ar[dr]_{{p^T}}
		\ar[r]^{U^{\hat{T}}}
		\ar@/^2pc/[rr]^{U^{\hat{T}}}
		&
		\params \times \carriers
		\ar[r]^{\hat{T}}
		\ar[d]_{\pi_\params}
		&
		\params \times \carriers\lltwocell<\omit>{<3>\beta}
		\ar[dl]^{\pi_\params} \\
		& \params
		}\]
	then we only need to show that it, together with $U^{\hat{T}}$, for any given $g : \clG\to\params$ induces a natural isomorphism between 1-cells in $H : g \to p ^T$ and $\hat{T}$-modules $(M : g\to\pi_\params,\lambda : \hat{T}\circ M\Rightarrow M)$. On one hand, each $H$ induces the module $(U^{\hat{T}}H, H\ast\beta)$ by composition and whiskering. On the other, a given module $(M,\lambda)$ for each $G$ and $MG=(\param_{G},\carrier_{G})$ defines a morphism $\lambda_G=(\id,\lambda_G^\carriers) : (\param_{G},T_{\param_{G}}\carrier_{G})\mapsto(\param_{G},\carrier_{G})$; then we can define a functor $H : G\mapsto (\param_{G},\carrier_{G},\lambda_G^\carriers)$. One can check that each of the two constructions provides a 1-cell in $\Cat/\params$, and they are one inverse to the other by definition of $\beta$. Everything is natural in $g$.
\end{proof}

\begin{remark}\label{remark_on_endofu}
	Analogously to \autoref{theorem:_product_presentation}, parametrized endofunctors and parametrized pointed endofunctors can be characterized internally in $\Cat/\params$. For  parametrized endofunctors $F  : \params \too \en\carriers$ with no extra structure, there is an evident way of constructing the total category of the fibration of endofunctor algebras as the category of endofunctor algebras of the endofunctor
	\[\vxy[@R=0cm]{\hat F :\clA\times\clX \ar[r]& \clA\times\clX\\
			(A,X) \ar@{|->}[r] & (A,F_AX)}\]
	as an endo-1-cell of $(\params\times\carriers,\pi_\params)$ in the $2$-category $\Cat/\params$, and the $2$-categorical characterisation of $\algTotal[F]{\params}{\carriers}$ is that of an \emph{inserter} \cite[(4.1)]{kelly_1989} of the identity and $\hat{F}$ in $\Cat/\params$.
\end{remark}
\autoref{theorem:_product_presentation} paves the way to an additional characterisation theorem for parametrized monads (that easily dualizes to comonads and adapts to more general parametrized endofunctors).
\begin{theorem}\label{as_monads_in_cokl}
The following pieces of data are equivalent:
\begin{enumtag}{pm}
\item \label{pm_1} a parametrized monad $T : \clA\times\clX \too \clX$, \ie (upon currying) a functor $\clA\too\Mnd{\clX}$;
\item \label{pm_2} a monad $T : (\clA\times\clX, \pi_\clA) \too (\clA\times\clX, \pi_\clA)$ in the 2-category $\Cat/\clA$, over the object $(\clA\times\clX,\pi_\clA)$;
\item \label{pm_3} a monad $T : \clA\times\clX \ptoo \clX$ in the 2-coKleisli category of strict 2-algebras, for the 2-comonad $\clA\times\firstblank$, over the object $\clX$.
% \item \label{pm_4} a monad in the (domain of the) 2-fibration $\Cat \ltimes^\ell \Cat$ obtained in \autoref{yoink_a_2_fibration}.
\end{enumtag}
\end{theorem}
\begin{proof}
Clearly, a monad like in \ref{pm_1} is an endo-1-cell of $\clX$ in the 2-category of \ref{pm_3}. Now it is enough to observe that the unit and multiplication of a parametric monad $T : \clA\too\Mnd{\clX}$ can be seen as 2-cells $\eta : \pi_\clX\To T$ and $\mu : T\bullet T \To T$ where $\pi_\clX : \clA\times\clX \too\clX$ is the projection functor and
	\[\label{nobody_expects_cokl_0} T\bullet T :
\clA\times\clX \xto{\Delta\times\clX}
\clA\times\clA\times\clX \xto{\clA\times T}
\clA\times\clX \xto{T} \clX\]
is a coKleisli composition.

As for \ref{pm_2}, note that we have to require $T$ to be a monad in $\Cat/\params$ and not in \emph{fibrations} over $\params$, because the adjunction $\params\times\carriers \leftrightarrows \algTotal[T]\params\carriers$ splitting the monad in \ref{pm_2} isn't always internal to fibrations (not all Cartesian arrows are preserved). 
\end{proof}
An immediate corollary of \autoref{as_monads_in_cokl} is the following.
\begin{remark}[Distributive laws between parametric functors]
	\autoref{proposition:_limits}, and \ref{as_monads_in_cokl} even more, are quite useful in pinpointing the correct notion of \emph{distributive law} between parametric endofunctors $S,T$%:\clX\too \clX$
    . 
    Recall that the classical notion of distributive law consists of a `lax intertwiner' $\lambda : ST\To TS$ between functors $S$ and $T$, so that it is compatible with the respective structure of $S$ and $T$ (them being monad and monad; comonad and monad; monad and comonad, etc.; cf.\, \cite[Appendix C.1]{Aguiar2020} for a thorough discussion on the matter of intertwining $p$ monads and $q$ comonads).

	Now \autoref{as_monads_in_cokl} allows describing distributive laws between parametric moands as distributive laws between a monad $T : \clA\times\clX \too\clA\times\clX$ on $\pi_\clA$, and a monad $S : \clB\times\clY\too \clB\times\clY$ on $\pi_\clB$. Though perhaps unpleasant, unwinding what this latter characterisation boils down to is just a matter of bookkeeping the notation.
\end{remark}
\autoref{as_monads_in_cokl} yields a concise characterisation of the simple fibration as a coKleisli object: a fibration $\var[p]\clE\clA$ is equivalent to $\klalg\clA\clX$ (resp., $\coklalg\clA\clX$) if and only if it is the (co)Kleisli object of a monad fibred over the projection.

An immediate consequence of such a characterisation is that 
\begin{remark}[The universal property of the simple fibration]
The simple fibration of \autoref{fib_reg_rep} and \autoref{fib_reg_rep} of a Cartesian category has a universal property in the 2-category $\Cat/\clA$; it is the coKleisli object of the comonad
	\[\vxy[@R=0cm]{
(A,X)\ar@{|->}[rr] && (A,A\times X) \\
\clA\times\clA\ar[rr]\ar[dddddr]_{\pi_\clA} && \clA\times\clA \ar[dddddl]^{\pi_\clA}\\
&&\\&&\\&&\\&&\\
&\clA &
}\]
\end{remark}
\begin{remark}
We have a complementary intuition on this construction: mimicking what happens when a monoid $M$ acts on itself on the left under the map $m\mapsto \lambda x.\,mx$, we call $P$ of Example~\ref{fib_reg_rep} the \emph{regular representation} of the cartesian monoidal category $(\clX,\times)$ on itself. In fact, one can easily generalize the example with this latter, more algebraic, intuition in mind.
\end{remark}
\begin{remark}
Monadicity and \autoref{remark_on_endofu} allow for swift characterizations of op/fibrations of co/algebras as universal objects; for example, the total Mealy opfibration of \autoref{example_mealy} comes up as the object of $R$-coalgebras in $\Cat/(\clK^\op\times\clK)$ for the aforementioned endofunctor $R$, and thus as the inserter
\[\vxy{
		\Mly \drtwocell<\omit>{} \ar[r]^-V\ar[d]_V& (\clK^\op\times\clK)\times\clK \ar@{=}[d]\\
		(\clK^\op\times\clK)\times\clK \ar[r]_-R & (\clK^\op\times\clK)\times\clK.
	}\]
The discussion in this section and \autoref{theorem:_product_presentation} also form the basis of comparing arbitrary fibrations against fibrations of algebras. A fibration $(\fibTotal,p)$ is an \textEM fibration if and only if there exists a morphism $(\fibTotal,p) \too (\params \times \carriers, \pi_\params)$ in $\Cat/\params$ which is \emph{monadic}. In \autoref{theorem:_the_fibration-monad_adjunction}, we shall see how, given a suitably nice fibration $p$, one can construct the appropriate fibration of algebras to compare it against.
\end{remark}

\subsection{The semidirect product intuition, and actegories}\label{subsec:act}

First of all, let us begin with a few results framing known results of group and monoid theory in our language. We believe they substantiate our intuition of $\emalg\Grp\Grp$ as a higher-dimensional version of a semidirect product: the semidirect product operation happens to be a functor with domain $\emalg\Grp\Grp$, and furthermore a left adjoint.

\begin{proposition}\label{ltimes_from_ltimes_grp}
	Let $G$ be a group, and $\alpha : G\times G\to G$ be the conjugation action $(g,h)\mapsto g^{-1}hg$ of $G$ on itself.
	Then, there is a functor
	\[
		\vxy{
			r_\alpha : \Grp \ar[r] & \emalg\Grp\Grp
		}
	\]
	sending $G$ to the object $(G,\alpha)^G\in\emalg\Grp\Grp$. The functor $r_\alpha$ has a left adjoint, sending a $G$-group $(H,\psi)^G$ to the \emph{semidirect product} $G\ltimes_\psi H$.
\end{proposition}
\begin{proof}
	A group homomorphism $\varphi : G\to H$ defines a morphism $(\varphi,\varphi) : (G,\alpha_G)^G\to (H,\alpha_H)^H$ of algebras, because the diagram
	\[\vxy{
		G\cdot G \ar[d]_{\alpha_H} \ar[r]^-{\varphi\cdot\varphi}& H\cdot H\ar[d]^{\alpha_H} \\
		G \ar[r]_\varphi & H
		}\]
	commutes, if $\varphi\cdot\varphi$ is defined as $G\cdot G \xto{G\cdot\varphi} G\cdot H \xto{\varphi\cdot H} H\cdot H$. This defines $r_\alpha$ on morphisms, and clearly composition and identities are preserved.

	The left adjoint to $r_\alpha$ is now defined in terms of the semidirect product operation of groups; let's show that $(H,\psi)^G\mapsto G\ltimes_\psi H$ is indeed a functor $\emalg\Grp\Grp\too\Grp$:
	\begin{itemize}
		\item morphisms in $\emalg\Grp\Grp$ are pairs of maps $(u,f) : (H,\psi)^G\to (K,\theta)^{G'}$ such that
		      \[\label{mor_in_emalg_Grp}\vxy{
				      G\times H\ar[d]_\psi\ar[r]^{u\times f} & G'\times K \ar[d]^\theta\\
				      H \ar[r]_{f} & K
			      }\]
		      Every such morphism induces a function $u\ltimes f : G\ltimes_\psi H\to G'\ltimes_\theta K$ between the semidirect products, defined sending $(g,h)\in G\ltimes_\psi H$ to $(ug,fh)$; evidently the identity of $G\ltimes_\psi H$ is preserved, and given how the composition operation is defined this is a group homomorphism:
		      \begin{align*}
			      (u\ltimes f)((g_1,x)\cdot(g_2,y)) & = (u\ltimes f)(g_1.g_2,\psi(g_2,x).y) \\
			                                        & =(u(g_1.g_2),f(\psi(g_2,x).y))        \\
			                                        & =(u(g_1).u(g_2),f(\psi(g_2,x)).fy)    \\
			                                        & =(u(g_1).u(g_2),\theta(ug_2,fx).fy)   \\
			                                        & =(ug_1,fx)\cdot(ug_2,fy)
		      \end{align*}
		\item composition and identities in $\emalg\Grp\Grp$ are preserved by this definition.
	\end{itemize}
	Now, to prove that there is an adjunction $\_\ltimes\_ \dashv r_\alpha$, one can establish the bijection
	\[\emalg\Grp\Grp\big((H,\psi)^G,(G',\alpha)^{G'}\big)\cong
		\Grp(G\ltimes_\psi H, G')\]
	natural in all its arguments. This essentially follows from how a morphism $(u,f) : (H,\psi)^G\to (K,\theta)^{G'}$ is defined in \eqref{mor_in_emalg_Grp}.
\end{proof}

\begin{remark}\label{ltimes_from_ltimes_mon}\label{remark:copower_action_of_monoids}
	In a similar vein, the semidirect product \emph{of monoids} \cite[3.9]{Janelidze+2003+99+114}, \cite[3.2]{Bourn1998} can be characterised as a functor
	\[
		\vxy{
			\emalg\firstblank\firstblank :\algTotal\Mon\Mon \ar[r] & \Mon \,,
		}
	\]
	having domain the total category of the \textEM fibration of the parametrized monad $T  : \Mon \to \Mnd{\Mon}$, where
	\[
		T_A X  \coloneqq A \cdot X = \sum_{a \in A} X
	\]
	acts by copowering a monoid $X$ with the set $A$, and $T_A$ gains its monad structure from the monoid structure on $A$.
 
	This functor preserves all colimits, as the semidirect product $M\ltimes_\phi N$ obtained from a representation $\phi : M\times N \to N$ via monoid homomorphisms is the coequalizer of the diagram
	\[
		\vxy{
			M\times N \ar@<.4em>[r]\ar@<-.4em>[r]& M*N \ar[r] & M\ltimes N
		}
	\]
	where the upper map is defined as $(m,n)\mapsto [m,a(m,n)]$ (seen as equivalence class of a word in $M*N$) and the lower map as $(m,n)\mapsto [n,m]$.

	It can be shown that $\emalg\Mon\Mon$ is locally presentable (using techniques from \cite{MP89}), and thus $\barealg\firstblank\firstblank$ is a left adjoint. Note how it might be difficult to establish its right adjoint explicitly, as the conjugation representation of \autoref{ltimes_from_ltimes_grp} doesn't exist for monoids.
\end{remark}

Notice that we construct $\emalg\Grp\Grp$ as in \autoref{remark:copower_action_of_monoids} starting from parametrized monad that acts via copowering. It can be similarly said for the total category $\Cat \ltimes^\ell \Cat$ of all parametrized monads (cf.  \autoref{Cat_ell_cat}), which allows us to consider the 2-functor
\[
	\Cat \ltimes^\ell \Cat \xrightarrow{\ltimes} \Cat
\]
that maps a parametrized monad to the total category of the fibration of \textEM algebras. At an informal level, we can consider this as an instance of the `micro-macrocosm principle': a construction of a certain type (the semidirect product of monoids) arises as a functor defined from a domain of a similar `semidirect product' form. 

With this intuition in mind, we now broaden our perspective in order to grasp the full extent of how things move from dimension 1 (of monoids and groups) to dimension 2 (of categories). We begin by recalling that it was pointed out in \cite{BJK:internal_object_actions} that, provided that $\params$ admits coproducts, there is an equivalence of categories\footnote{If $\clX$ admits enough colimits so that $\Mnd\carriers$ has coproducts, this latter category is in turn equivalent to the category of coproduct preserving functors $\params^\coprod\too\Mnd\carriers$, where $\params^\coprod$ is the free coproduct completion of $\params$, but we will not need this further characterization. Another characterization that we will not need is that these equivalence of categories are in fact equivalences of 2-categories.} between
\begin{itemize}
	\item functors of type $\params\too\Mnd\carriers$,
 	\item lax monoidal functors of type $(\params,+)\too\en\carriers$.
\end{itemize}
In the language of \cite{075b1350ce224f5e9f9c4e2642823a04}, lax monoidal functors of this kind would also be called \defn{lax }$(\params,+)$\defn{-actegories}.

The equivalence is easy to describe explicitly: given a functor $F  : \params \too \Mnd{\carriers}$, one defines laxators
\[
	F_A\circ F_B \xRightarrow{\iota_A \circ \iota_B}  F_{A+B}\circ F_{A+B} \xRightarrow{\mu^{A+B}}  F_{A+B}
\]
and a unitor $\eta^\varnothing : \id\carriers\To F_\varnothing$ is given by the unit of the monad $F_\varnothing$. Vice versa, any lax monoidal functor $F  : (\params,+)\to (\en\carriers ,\circ)$ lets us consider each $F_A$ as monad with multiplication
\[
	\vxy{F_A\circ F_A \ar@{=>}[r] & F_{A+A} \ar@{=>}[r] & F_A.}
\]
Now, each monoidal category $(\params,\otimes)$ equips the endo-2-functor $\params \times \firstblank$ of $\Cat$ with the structure of a $2$-monad, the (lax, pseudo, strict) algebras of which are precisely the (lax, pseudo, strict) $(\params,\otimes)$-actegories. This is a banality: monoidal categories are precisely pseudomonoids in $\Cat$; also, this is just the 2-dimensional, lax analogue of Example~\ref{example_2}.

Restricting our attention to coCartesian monoidal categories, this point of view allows us to see the category of parametrized monads as the 2-category $\coCartCat \ltimes^\ell \Cat$ of lax actegories over coCartesian monoidal categories, corresponding to the parametrized $2$-monad
\[
	\coCartCat \to \text{2-}\Mnd{\Cat} \,, \quad \clA \mapsto \clA \times \firstblank
\]
that the cocartesian monoidal structure on the categories induces.
\color{red}
\color{black}
\begin{remark} \label{remark:_comonad_algebras}
	One could get around the annoyance of being restricted to coCartesian categories of parameters $\params$ by considering the following variation of the construction.
	The coCocartesian monoidal structure $+:\params \times \params \to \params$ is left adjoint to the (essentially unique) pseudocomonoid structure given by the diagonal $\Delta : \params \to \params \times \params$, which will induce the parametrized $2$-comonad analogue of Example~\ref{fib_reg_rep} on $\params \times \firstblank$.

	Due to $+$ and $\Delta$ being adjoint, the lax algebras of the parametrized $2$-monad $\params \times \firstblank$ are in bijection with the lax \textEM coalgebras of the parametrized $2$-comonad $\params \times \firstblank$. While the use of lax algebras instead of lax coalgebras for a $2$-comonad is a non-standard notion, in this example these lax algebras correspond precisely to parametrized monads, with the additional benefit that the $2$-comonad structure on $\params \times \firstblank$ doesn't require $\params$ to be cocartesian.
\end{remark}

\begin{remark}[On the relation with graded monads]\label{rel_grad_mona}
	Lax monoidal functors from a one-object bicategory are sometimes also called \emph{graded monads}; they have their own notion of algebra, different from ours (cf. \cite[Definition 5.1]{Dorsch2018GradedMA}). The category of algebras of a lax monoidal functor $(\params,+) \to \en\carriers$, considered as a graded monad, is calculated as the \emph{lax limit} of the lax functor $(\params,+) \to \Cat$, which has no evident associated fibration over $\params$.

	A comparison between graded monads and parametrized monads (and thus, implicitly, with our framework) can be found in \cite{fujii20192categorical}. In a few words, our theory of fibrations of \textEM algebras relates with the fairly well-developed theory of graded monads as follows:
	\begin{itemize}
		\item we do not require the representation $T : \params \too \Mnd\carriers$ to be monoidal, \ie we do not require that $T$ lifts to a lax functor from a bicategory --a mere category suffices;
		\item we do not consider algebras for all parameters at once, as in \cite{milius_et_al:LIPIcs:2015:5538,Dorsch2018GradedMA}, but instead for each object separately.
	\end{itemize}
	The second point deserves more explanation and sheds light on the precise comparison between the two concepts: using a technique from \cite[Definition 1.2]{Ozornova2021} it is possible to define a certain bicategory $\Sigma\params$ (circumventing the absence of a monoidal product in $\params$) called the \emph{suspension} of $\clA$, such that parametrized monads $T : \params\times\carriers\too\carriers$ are classified by lax functors $\bar T$ out of $\Sigma\clA$:
	\[\Cat(\params,\Mnd\carriers)\cong \textbf{Lax}(\Sigma\params,\Cat)\]
	given a diagram $\params^\op\to \Cat : A\mapsto \EM(T_A)$, the lax limit of $\bar T : \Sigma\params\to\Cat$ is precisely $\algTotal[\EM]{\clA}{\clX}$. 
 
 This note does not further analyze the consequences of this remark, as no example or application gets easier when $T$ is replaced by $\bar T$; but in a forthcoming, more comprehensive work, we plan to delve deeply into the two-dimensional aspects of this topic.
	% \greta{By this we mean that we are better because here there is? Even if that's not the case, we need a bottom line here.}
	% \fosco{Heh, I am not convinced either: in my view, graded monads and diparametrized ones go together as `comparison with what other people have interpreted as parametric stuff'}
	% \greta{Fair enough, then I think we should consider moving this somewhere else (even if it's in the introduction, and we say `this is different from...'). We should do it when the overall structure is definitive, though.}
\end{remark}
%\ulo{In the above, perhaps make things more clear (the total category of the fibration of algebras is the lax limit, not the pseudofunctor $\clA^\op \too \Cat$.)}

\section{Completeness and cocompleteness results}
While our definitions comprise many examples, the reader may not yet be convinced of the fruitfulness of approaching algebras using tools from the theory of fibrations. In this section we heavily use it in order to prove results about how to compute limits (cf. \autoref{subsec:lims}), about adjoints to reindexing and cocompleteness (cf. \autoref{subsec:adj_to_reidx}), and about the existence of certain colimits, and accessibility issues (cf. \autoref{subsec:colims}).

%\todo[inline]{Fouche: metto una sezione nuova come separatore; da un punto di vista filosofico l'esistenza di aggiunti ai reindexing è un teorema di co/completezza, perché l'unico modo in cui lo dimostriamo è appellarci a un aft. Per il resto: invece l'intuizione come prodotto semidiretto è ottima ma non c'entra con i colimiti}

\subsection{Limits in fibrations of algebras}\label{subsec:lims}

From \autoref{theorem:_product_presentation}, it immediately follows that,
\begin{proposition} \label{proposition:_limits}
  Let $T : \params\too \Mnd{\carriers}$ be a parametric monad. Then, the forgetful functor
	\[
    \algTotal[T]{\params}{\carriers} \to \params \times \carriers , \quad (\carrier,\xi)^A \mapsto (\param,\carrier)
	\]
	is monadic, and therefore it creates limits.
\end{proposition}
As a corollary, we get an explicit way to compute limits in $\algTotal[T]\params\carriers$.

\begin{remark}%\label{explicit_way_limits}
  %\color{blue}
	%Recalling how limits in \textEM categories of algebras are created, we now know how to calculate limits in $\algTotal{\params}{\carriers}$ using limits in $\carriers$ and in $\params$, something which is not immediate from the approach that uses the Grothendieck construction. 

  %The recipe to build limits goes as follows. 
  Let $D : \clJ\too \algTotal\params\carriers$ be a functor; its components are given by  $DJ= \alg{\param_J}{\carrier_J}{\xi_J}$. From this we can consider the limit $\param = \lim_J \param_J$ of all parameters, and from the terminal cone maps $\alpha_J : \param \too \param_J$ we obtain the algebras $\alpha_J^*(\alg{\param_J}{\carrier_J}{\xi_J})$, each of which is a $T_\param$-algebra. 
   Examining what we get, we find that the limit $\alg{\param}{\carrier}{\xi}$ of these objects \emph{computed in $\EM(T_\param)$} yields the limit of $D$.
   % (this can be verified by comparing to the limit we get from \autoref{proposition:_limits}). The same argument works for endofunctor algebras.% as well.
   %Considering how $\algTotal[F]\params\carriers$ and $\EltsAlg\carriers$, 
%    $\EltsEM\carriers$ relate via \eqref{pb_endo}, \eqref{pb_em}, a similar argument shows also how to compute limits in a fibration of algebras modelled on $F : \clA \too \en\clX$: given a diagram $\barealg\clA\clX$ with components $(A_J; X_J, x_J)$, compute the limit $A = \lim_J A_J$ of the parameters; reindex all $F_{A_J}$-algebras to make them $F_A$-algebras and compute the limit in that fibre.
\end{remark}

\subsection{Adjoints to reindexing, bifibrations and colimits}\label{subsec:adj_to_reidx}

Recall that a fibration is a bifibration if all the reindexing functors have left adjoints \cite[9.1.2]{CLTT}. While the reindexings of a fibration of algebras need not in general have left adjoints, they preserve all types of limits that exist in $\carriers$. To see that this is the case, simply observe that the reindexing functors make the triangles
\[\vxy{
	\Alg(G)
	\ar[rr]% ^{\hat{F}}
	\ar[dr] %_{\pi_\params}
	& & \Alg(F)
	\ar[dl] %^{\pi_\params}
   \\
	& \carriers &
	}\]
involving the (limit creating) forgetful functors commute. With the adjoint functor theorem in mind, therefore, a fibration of algebras is intuitively never very far from being a bifibration, the main obstruction being the lack of colimits in the fibres.

If we are dealing with \textEM algebras, we can make use of the standard fact that any comparison functor $\EM(T) \too \EM(S)$, induced by a monad morphism $S \To T$, has a left adjoint as soon as $\EM(T)$ has reflexive coequalizers.  

\begin{proposition} \label{prop:monad_algebra_bifibration}
Given a parametrized monad $\params \too \Mnd{\carriers}$, if all the fibres of the fibration $p^T \colon \algTotal{\params}{\carriers} \fib \params$ have reflexive coequalizers, then $p^T$ is a bifibration. 
\end{proposition}

In particular, if $\carriers = \Set$, the fibration of \textEM algebras is a bifibration.
Note that under the assumptions of \autoref{prop:monad_algebra_bifibration}, if the category $\carriers$ is additionally cocomplete, then Linton's construction of colimits of algebras \cite{Linton_coeqs} implies that the fibres of the fibration of Eilenberg-Moora algebras are all cocomplete as well.

% \begin{proposition}
% 	\label{prop:pres}Let $\clJ$ be a small category, and assume $\carriers$ admits $\clJ$-shaped limits; then the reindexing functor $\alpha^* : \Alg(G)\too \Alg(F)$ preserves limits of shape $\clJ$ for every $\alpha:F\To G$ in $\en\clX$.
% \end{proposition}

% \[\vxy{
% 	\Alg(G)
% 	\ar[rr]^{\hat{F}}
% 	\ar[dr]_{\pi_\params}
% 	& & \Alg(F)
% 	\ar[dl]^{\pi_\params} \\
% 	& \carriers &
% 	}\]
% While 

% \begin{corollary}
% 	If $\clX$ is complete then $\EltsAlg{\clX}$ is complete too and $\var[U]{\EltsAlg{\clX}}{\en\clX}$ preserves all limits.
% \end{corollary}

When dealing with endofunctor algebras, in order to make the situation more tractable, we start with the classifying case, namely with fibrations as in Example~\ref{subexam_algebras}, with some additional properties on the category $\clX$ (=that it is $\kappa$-accessible) and restricting to the functors $F\in\en\clX[<\kappa]$ (=the subcategory of $\kappa$-accessible functors), it is possible to employ a specific form of the adjoint functor theorem to prove the existence of left adjoints for each reindexing.
\begin{theorem}\label{swindle_theorem}
	Let $\clX$ be $\kappa$-presentable and assume that the fibration of algebras is restricted to just the $\kappa$-accessible functors $\clX\too \clX$, then each reindexing $\alpha^*$ has a left adjoint $\opFibReindex{\alpha}$.
\end{theorem}
\begin{proof}
	The proof follows an argument that we would like to call `Freyd swindle';%
	\footnote{The term `swindle' entered some parts of mathematical practice as a non-derogatory way to refer to `clever tricks, akin to sleights of hand, providing proofs of true statements based on illegal but evocative manipulations'. Example of this might be the offhanded re-bracketing of an infinite sum of real numbers (as in Euler's proofs of some analytic identities), or an infinite direct sums of $R$-modules, as in \cite[Corollary 2.7]{lamlam} where the term `Eilenberg swindle' was apparently coined, or of connected sums of compact manifolds, \cite{Tao_2009}.

		From our category-theoretic standpoint, the sleight of hand is instead a demonstration of stubbornness: to build a certain universal object, you apply a certain universal construction, for example a pushout. The pushout will hardly compute the desired left adjoint $\sum_\alpha(A,a)$. But one insists, and applies the same construction again, to the new piece of data; the pushout at this second stage will hardly be the desired $\sum_\alpha(A,a)$. But you insist, \dots{} and after a certain number of steps you get an answer in the form of a fix-point: in our case, $SP_\infty=S\left(\colim \dots\right)\cong\colim\, S\dots$ is the clever re-bracketing (but apart from a few little sins of omissions in what exactly the isomorphism means, the argument is totally rigorous).

		As for attributing the swindle to P.J. Freyd, more than one bright argument in his mathematical work is based on a similar technique of `doing one thing for a transfinite number of steps'.%
	} we sketch the idea in the case $\kappa=\omega$, for bigger ordinals one argues similarly.

	Let $\alpha^* : \Alg(G) \too \Alg(F) : \alg{G}{Y}{\zeta} \mapsto \alg{F}{X}{\alpha_X \circ \zeta}$ be the functor between algebras induced by a natural transformation $\alpha : F \To G$, and consider an $F$-algebra $\xi : FX\to X$; consider the pushout
	\[\vcenter{\xymatrix{
				FX \xpo \ar[r]^\alpha \ar[d]_{\xi} & GX \ar[d]\\
				X \ar[r]_{t_0}& P_0
			}}\]
	and define inductively the chain at the lower horizontal side of the diagram
	\[\vcenter{\xymatrix{
				FX \xpo \ar[r]^\alpha \ar[d]_{\xi} & GX \ar[r]^{Gt_0}\ar[d]\xpo & GP_0 \ar[r]^{Gt_1}\ar[d]\xpo& GP_1 \ar[d]\ar[r]&\dots \ar[r]\xpo & \boxed{?}\ar[d]\\
				X \ar[r]_{t_0}& P_0 \ar[r]_{t_1} & P_1 \ar[r]_{t_2} & P_2 \ar[r]_{t_3}&\dots\ar[r] & P_\infty
			}}\]
	The accessibility assumption on $G$ now implies that $GP_\infty$ is the colimit of the upper horizontal chain, so $\boxed{?}=GP_\infty$. The pushout cocone then yields a canonical choice of a map $\zeta_\infty : GP_\infty\to P_\infty$.
\end{proof}
% \begin{remark}
% 	Note that each reindexing having a left adjoint is equivalent to $U$ being a bifibration.
% \end{remark}

\subsection{Colimits in fibrations of algebras}\label{subsec:colims}

%\ulo{I separated the adjoints to reindexing stuff from colimit stuff into separate subsections. TODO: haven't adjusted the following section yet to account for this change.}

%\ulo{Really, the adjoints to reindexing and the colimit stuff should probably be in its own section not in the \textit{Alternative approaches to fibrations of algebras} section.}

In Proposition~\ref{proposition:_limits} we obtained that limits in $\algTotal{\params}{\carriers}$ are created by a monadic functor $\algTotal{\params}{\carriers} \too \params \times\carriers$. Colimits in categories of algebras, on the other hand, are usually way more difficult to compute as, even when they exist, they tend to be complicated objects -- this is the case, for example, for initial endofunctor algebras.

The present subsection deeper investigates the two following connected problems:
\begin{itemize}
	\item
	      how to induce/compute colimits in the total categories of the \textEM or endofunctor algebra fibrations?
	\item
	      how to ensure that reindexings (which preserve all limits, given the fact that these are created in the base) have left adjoints?
\end{itemize}
A partial answer was already given in \autoref{subsec:adj_to_reidx}, but we find that the second problem can be addressed by looking at the `classifying' fibration of algebras and then argue that some relevant property is preserved under pullback. In fact, we leverage on the following result.
\begin{lemma}
	Let $\var[p]{\fibTotal}{\params}$ be a fibration such that each reindexing $\fibReindex{u} : \fibTotal_{\param'} \too \fibTotal_\param$ preserves (co)limits; let $F : \params' \too \params$ be a functor and $q$ the fibration obtained pulling back $p$ along $F$. Then, each reindexing of $q$ also preserves (co)limits.
\end{lemma}
Therefore all fibrations obtained by pulling back the classifying fibration of algebras share its properties in terms of limit-preserving reindexings.

Let us start on our pursuit with the case of monads.
Based on \cite[4.3.6]{Bor2} and on the fact that  $\algTotal\clA\clX$ is monadic over $\clA \times \clX$, one can prove very swiftly the following.
\begin{theorem}
	Let $T : \clA\too\Mnd{\clX}$ be a parametrized monad such that $T$ preserves filtered colimits in each variable separately; if $\clX$ is cocomplete, then so is $\algTotal\clA\clX$.
\end{theorem}

\begin{remark}
	In order to proceed with the discussion, let's make blanket assumptions the ones in \autoref{swindle_theorem}, \ie let's restrict the action of the pseudofunctor $F\mapsto \Alg(F)$ to accessible functors (and monads) defined on an accessible base category $\carriers$. More formally, we consider the accessible analogue of \eqref{pb_endo} and \eqref{pb_em}, \ie the pullbacks
	\[\vxy{
		(\algTotal{\en\carriers}{\carriers})_{<\kappa}\ar[r]\ar[d]_{U_{<\kappa}}\xpb & \algTotal{\en\carriers}{\carriers}\ar[d]^U &
		(\algTotal{\Mnd\carriers}{\carriers})_{<\kappa}\ar[r]\ar[d]_{U_{<\kappa}}\xpb & \algTotal{\Mnd\carriers}{\carriers}\ar[d]^U\\
		\en\carriers[<\kappa] \ar@{^{(}->}[r]& \en\carriers &
		\Mnd\carriers_{<\kappa} \ar@{^{(}->}[r]& \Mnd\carriers \\
		}\]
\end{remark}

We are then left with the problem of establishing how colimits in each fibre $\falgFibre \param \carriers$ of a fibration of \textEM algebras for $T$ induce, if anything, \emph{global} colimits in the whole $\barealg\params\carriers$, provided some additional assumptions on $T$ are made. The issue is somewhat subtle, as it is well-known that being `internally' cocomplete as an object in $\Fib$ is a weaker property than having a cocomplete total category, \cite[p. 87]{streicher2023fibred}. Are total categories of fibrations of algebras cocomplete in the weak sense, or in the strong sense? And how does this relate to assumptions made on $\params,\carriers$ or $T : \params\times\carriers \too \carriers$?

One preliminary remark is the following well-known observation by Linton \cite{Linton_coeqs}: an \textEM category over a cocomplete base has coproducts as long as it has certain specific reflexive coequalisers, and it is cocomplete if it has all reflexive coequalisers.
Computing such coequalisers is a difficult task in general, although a general recipe has been given by Linton \cite{Linton_coeqs}.
Calculations applying this general result provide us with the following.
\begin{proposition}\label{prop:fib_for_colim}
	In the specific case of $\algTotal\params\carriers$ regarded as an \textEM category (cf. \autoref{theorem:_product_presentation}), in the computation of binary coproducts, the coequaliser in question happens in a single fibre -- the one over  $A+B$, the coproduct in $\params$.
\end{proposition}
\begin{proof}
	Letting $*$ denote the coproduct of algebras, computing coequalisers in the specific case of $\algTotal\params\carriers$ regarded as an \textEM category goes as follows: given algebras $(X,\xi)$ and $(Y,\theta)$, their coproduct is the coequaliser of the reflexive pair
	\[\vxy[@C=1.4cm]{
		FUFU(X,\xi) * FUFU(Y,\theta) \ar@<.4em>[r]^-{\epsilon FU(X,\xi) * \epsilon FU(Y,\theta)} \ar@<-.4em>[r]_-{FU\epsilon_X * FU\epsilon_Y} & FU(X,\xi) * FU(Y,\theta)
		}\]
	where the coproduct of free algebras is simplified by the observation that $FUX*FUY\cong F(UX+UY)$ since $F$ is a left adjoint.

	In the specific case of $\emalg\clA\clX$ the coequaliser that defines the coproduct $(X,\xi)^A*(Y,\theta)^B$ is then
	\[\vxy[@C=1.4cm]{
			\big(T_{A+B}(T_AX+T_BY),\mu^{A+B}\big)^{A+B} \ar@<.4em>[r] \ar@<-.4em>[r]& \big(T_{A+B}(X+Y), \mu^{A+B} \big)^{A+B}
		}\]
	where it is not fundamental to make explicit what objects are involved, but rather that the coequaliser in question happens in a single fibre, the one over $A+B$ (coproduct in $\clA$). This colimit can't in general be reduced further (although there are cases where it acquires a more explicit form, for example when $A$ is the initial object of $\clA$).
\end{proof}
In light of this, our fibre-wise approach seems well justified. As in the case of \autoref{swindle_theorem}, additional properties on the categories involved make the problem more tractable, so much so that we can prove the following.

\begin{theorem}\label{thm:accessible}
	Let $\clA,\clX$ be both $\kappa$-accessible categories for some regular cardinal $\kappa$; furthermore, let all the parametrized functors $T_A$ that we consider be $\kappa$-accessible.
	Then the pseudofunctor $\Alg_\clX : \en\clX[<\kappa]^\op \too \Cat$ is accessible in the sense of \cite[5.3.1]{MP89}, which means that
	\begin{itemize}
		\item the base category and each fiber are accessible,
		\item each reindexing is an accessible functor,
    \item $\Alg_\clX(\colim\, T_i)\cong \lim_i \Alg_\clX(T_i)$ for every $\kappa$-filtered diagram $\clJ \too \en\carriers[<\kappa]$.%. : i\mapsto T_i$.
	\end{itemize}
\end{theorem}
\begin{proof}
	% \fosco{Expand on this}
	The category of $\kappa$-accessible functors $\en\clX[<\kappa]$ is itself accessible (although for a bigger cardinal $\lambda\gg\kappa$; if $\clX$ is $\kappa$-\emph{presentable} instead, the category $\en\clX[<\kappa]$ is $\kappa$-presentable). Now, a simple way to get accessibility of all fibers is to appeal the results of \cite[Ch. 5]{MP89} (in particular, all 2-limits of accessible categories are accessible categories).

	Each reindexing is accessible, it follows from the assumption that all $T_A$ are accessible functors, by checking the universal property.

	As for the last condition, we have to prove that the canonical comparison functor
	\[
		\vxy{
			\Alg_\clX(\colim\, T_i) \ar[r] & \lim_i \Alg_\clX(T_i)
		}
	\]
	obtained from the cone $\alpha_i^* : T_i\To \colim\, T_i$ is in fact an equivalence. The proof is a matter of unwinding the definition and checking its universal property.

	This ensures that $\var[p]{\algTotal\params\carriers}\params$ is accessible (which means its total category is accessible, and its projection is an accessible functor).
\end{proof}
Note that the hypotheses are not as restrictive as they seem, as they allow us to cover in one swoop:
\begin{itemize}
	\item all cases in which reindexings are covariant: co\textEM opfibrations, Kle\-i\-sli opfibrations, endofunctor coalgebras\dots; this can also be appreciated directly, relying on the completeness theorem for the 2-category of accessible categories, and presenting the the relevant opfibrations as pullbacks of accessible functors over $\en\clX[<\kappa]$;
	\item the fibrations of endofunctor algebras over accessible categories, cf.\  \autoref{swindle_theorem};
	\item \textEM fibrations: notably, the proof of \autoref{swindle_theorem} can be carried over unchanged, and $s_\infty$ will be an \textEM algebra assuming $S,T$ are monads. But there is a slicker argument in that case: from \cite[4.3.2]{Bor2} we know that a category $\clC^T$ of \textEM algebras for a monad $T : \clC\too\clC$ has all colimits that $T$ preserves. Then, it will immediately follow from the monadicity result that if $\clA\times\clX\too \clA\times\clX : (A,X)\mapsto (A, T_AX)$ preserves colimits of shape $\clJ$, then $\emalg\clA\clX$ has colimits of shape $\clJ$.
\end{itemize}

%% file: sections/4-comparing.tex
\section{Comparing fibrations against fibrations of algebras}\label{sec:comparing}
Our main goal in this section is to show how to build a parametrized monad out of a suitably nice fibration $p : \clE\fib\clB$, and explore some consequences of this result: we can then attempt to `measure' how far the original fibration was from being a fibration of algebras of its corresponding parametrized monad.

In doing so, we prove that the construction that associates a parametrized monad to $p$ is `free', in the sense of being a left adjoint to a suitable functor in the opposite direction, building the fibration of \textEM algebras of a parametrized monad. This sets up an adjunction
\[
	\vxy{
		\niceFib{\params} \pair{}{} & \ParMndInit{\params}
	}
\]
between a category of fibrations (that we call \emph{pruned} in \autoref{pruned_fib}) and a category of monads (that we also call \emph{pruned} in \autoref{pruned_monad}), and a fibration $p$ arises as a fibration of \textEM algebras if and only if the unit of this adjunction is an isomorphism at $p$.

Clearly then, fibrations of algebras arise as fixpoints of this construction, hence every fibration of \textEM algebras will be an example of pruned fibration, under mild assumptions on its category of parameters and of carriers.%; we start that by showing that it satisfies a few properties---these we will later be requiring from a general fibration `aspiring' to be a fibration of algebras.

The first observation that we need to carry on our analysis is that an \textEM fibration has full and faithful left and right adjoints under very mild conditions. The easiest way to see this is using the following standard fact about fibrations.

\begin{lemma}
	A fibration $p  : \fibTotal \fib \params$ has
	\begin{enumerate}
		\item a full and faithful left adjoint $\ladj{p}$ precisely when each of the fibres  $\fibTotal_\param$ contains an initial object $\initial_\param$ and
		\item a full and faithful right adjoint $\radj{p}$ precisely when each of the fibres $\fibTotal_\param$ contains a terminal object $\terminal_\param$ and the reindexings $\fibTotal_\param \to \fibTotal_{\param'}$ preserve terminal objects.
	\end{enumerate}
\end{lemma}

\begin{remark}
	There is a concise way to rephrase the above observation in terms of a property of $p$ as an object of the 2-category of fibrations over $\clA$: the Cartesian functor $p : p\too\id[\params]$ has a fibred left (resp., right) adjoint if and only if $p$, regarded as an object of $\Fib/\clA$ admits an initial (resp., terminal) object, which boils down exactly to an initial (resp., terminal) object in each fiber, preserved by reindexings. Cf. \cite[1.8.8]{CLTT}.% for this statement in the context of fibrations, and \cite{foo} for the general statement in a 2-category.
\end{remark}
\begin{proposition}
	Let $T  : \params \too \Mnd{\carriers}$ be a parametrized monad. If $\carriers$ has an initial and a terminal object, then each monad $T_\param$ will have an initial algebra $\initial^T_\param$ and a terminal algebra $\terminal^T_\param$ which allow us to define the full and faithful functors
	\[
		\ladj{p}  : \params \too \algTotal{\params}{\carriers}\,,\;\; \param \mapsto \initial^T_\param
		\quad \text{ and } \quad
		\radj{p}  : \params \too \algTotal{\params}{\carriers}\,, \;\; \param \mapsto \terminal^T_\param \,,
	\]
	that are respectively the left adjoint and right adjoint of the functor
	$p : \algTotal{\params}{\carriers} \too \params$.
\end{proposition}
\begin{proof}
	The initial and terminal algebra of $T_\param$ will respectively be given by the free $T_\param$-algebra on the initial object of $\carriers$ and the unique algebra on the terminal object of $\carriers$. As the reindexing of a fibration of algebras does not change the carrier, the terminal algebras will be preserved.
\end{proof}

We mostly make use of the left adjoint, but existence of the right adjoint is an important property which can help identity fibrations of algebras in nature.

% \begin{assumption}
% 	From now on all fibrations have fully faithful adjoints \fosco{on which side!? Both?}
% \end{assumption}
\begin{remark}
	The fibration of algebras of a parametrized endofunctor $F  : \params \too \en\carriers$
 will also have a right adjoint if $\carriers$ has a terminal object, and the algebra structure will still be created by the forgetful functor, but since the existence of initial algebras of endofunctors is a much more complicated question than for monads, then the existence of the left adjoint is very much dependent on the particular parametrized endofunctor.
\end{remark}

Since we are dealing with fibrations, there is a well-behaved sense in which we can take `kernels'.
\begin{construction}\label{the_initial_fibre}
	If the codomain of a fibration $p  : \fibTotal \fib \params$ has an initial object $\initial$, then we can consider the fibre $\fibTotal_\initial$ over $\initial$, which we will call the \emph{initial fibre} of $p$. We can interpret the fibre $\fibTotal_\initial$ along with its inclusion $i  : \fibTotal_\initial \to \fibTotal$ as being the \emph{kernel} of the fibration $p$.
\end{construction}
One of the reasons for taking the initial fibre over other fibres is the following.
\begin{proposition} \label{prop:adjoint_of_kernel}
	For any fibration $p  : \fibTotal \fib \params$ where $\params$ contains an initial object $\initial$, the inclusion $i  : \fibTotal_\initial \to \fibTotal$ of the initial fibre has a right adjoint $\radj{i}  : \fibTotal \to \fibTotal_\initial$.
\end{proposition}
\begin{proof}
	The right adjoint $i^R$ acts by mapping an object $E$, living in the fibre over $pE$, into an object $u_!^*E$ in the initial fibre by reindexing $E$ along the unique morphism ${u_! : \initial \to p(E)}$.
\end{proof}

We are now ready to express how we can turn fibrations into adjunctions, sliced over $\clA$ -- \ie in the slice category of $\Cat/\clA$: it is in general too restrictive to require that these adjunctions to be \emph{fibred} over $\clA$, which would require the preservation of cartesian lifts.

This section is based on setting up an adjunction between the total category of a fibration and a certain product of categories. The following lemma shows that 
it indeed involves quite a simple structure.

\begin{lemma} \label{lemma:functor_copairing}
    Given a span $\carriers \xleftarrow{u} \fibTotal \xrightarrow{v} \params$ of functors and the cospan $\carriers \xrightarrow{\ladj{u}} \fibTotal \xleftarrow{\ladj{v}} \params$ of their left adjoints, the induced functor
     $\pairing{v}{u} : \fibTotal \too \params \times \carriers$ has a left adjoint $\coPairSum{\ladj{v}}{\ladj{u}} \colon \params \times \carriers \too \fibTotal$ precisely when the coproduct $\ladj{v}(\carrier) + \ladj{u}(\param)$ in $\fibTotal$ exists for every $\carrier \in \Ob{\carriers}$ and every $\param \in \Ob{\params}$. Defining 
     \[
        \coPairSum{\ladj{v}}{\ladj{u}}(E) \coloneqq \ladj{v}(E) + \ladj{u}(E) 
     \]
     will give us the left adjoint
     $\coPairSum{\ladj{v}}{\ladj{u}} \dashv \pairing{v}{u}$.
\end{lemma}

Observe also that if $\pairing{v}{u} : \fibTotal \too \params \times \carriers$ has a left adjoint and $\params$ and $\carriers$ contain initial objects, the left adjoints $\ladj{v}$ and $\ladj{u}$ must necessarily exist.

The aim is to apply \autoref{lemma:functor_copairing} to the span $\carriers \xleftarrow{\radj{i}} \fibTotal \xrightarrow{p} \params$ which is built from a fibration $p$ using \autoref{prop:adjoint_of_kernel}. We will now define a \textit{pruned fibration} as one for which this process can be performed, including additional the assumptions, which will be used to prove \autoref{theorem:_the_fibration-monad_adjunction}.

\begin{definition}[Pruned fibration]\label{pruned_fib}
	The 2-category $\niceFib{\params}$ of \defn{pruned fibrations} is defined as the full 2-subcategory of $\Fib(\params)$ having for objects those fibrations $p : \fibTotal \fib \params$ for which the following conditions are satisfied:
	\begin{itemize}
		% \item the initial fibre $\clE_\initial$ (cf. \autoref{the_initial_fibre}) admits an initial and a terminal object;
		\item $p$ has a full and faithful left adjoint $\ladj p$;
        \item the coproduct $\initial_A + E_0$ in $\fibTotal$ exists for all $A \in \Ob{\params}$ and $E_0 \in \Ob{\fibTotal_\initial}$;
        \item $p$ preserves the above coproducts.
	\end{itemize}
	In particular, if $\fibTotal$ has coproducts, all fibrations $p$ admitting a full and faithful left and a right adjoint, denoted respectively $\ladj p$ and $\radj p$, are pruned.

The main goal of this section is to give an explicit process which turns a pruned fibration into a particular kind of parametrized monad.
%
%\ulo{TODO: Define the $2$-cat as the full subcategory of $\Cat \ltimes^\ell \Cat$ (as a fibre?)}
%
%\ulo{TODO: I think I recall seeing $\Cat \ltimes^\ell \Cat$ appear in a few places as $\Cat \ltimes \Cat$ without the superscript $\ell$.}
%
\end{definition}
\begin{definition}[Pruned monad]\label{pruned_monad}
	The 2-category $\ParMndInit{\params}$ of \defn{pruned monads} is defined as the full 2-subcategory of $\Mnd{\params}$ having for objects  the parametrized monads $T : \params\times\carriers\too\carriers$ that act trivially on the initial parameter (\ie, $T_\initial$ is the identity functor, or isomorphic to it) and whose carrier category has an initial object.
\end{definition}

The rest of the section is essentially an analysis of the construction that maps a pruned fibration $p$ to the monad generated by the adjunction  
$\coPairSum{\ladj{p}}{i} \dashv \pairing{p}{\radj{i}}$ (as in \autoref{lemma:functor_copairing}).

\begin{lemma} \label{lemma:_fibrations_to_adjunctions}
 %    Let $p \colon \fibTotal \fib \params$ be a pruned fibration. ...

 %    ....will be an adjunction between the projection
	% $\var[\pi_\params]{\params \times \fibTotal_\initial}{\params}$
	% and
	% $\var[p]{\fibTotal}{\params}$
	% in the slice-2-category $\Cat/\params$. In explicit components, $(p\times \radj{i})\circ \Delta=\langle p,\radj{i}\rangle$, and its left adjoint sends $(\param,E_0)$ to $\ladj{p}\param + iE_0$.
 %
 % {\color{orange}
 Let $\params$ be category with finite coproducts and an initial object, and let $p  : \fibTotal \fib \params$ be a fibration with a full and faithful left adjoint $\ladj{p}$. We then have the composite adjunction
	\[\label{composite_adj}
		\vxy{ \params \times \fibTotal_\initial
			\pair {\ladj{p} \times i}{p \times \radj{i}}
			& \fibTotal \times \fibTotal
			\pair {+}{\Delta}
			& \fibTotal
		}
		% \text{ or explicitly, }
		% \vxy{ \params \times \fibTotal_\initial \pair {(\param,E_0) \mapsto \ladj{p}\param + iE_0}{\langle p,\radj{i}\rangle} & \fibTotal\,, }
	\]
	which, if $p$ preserves finite coproducts, will be an adjunction between the projection
	$\var[\pi_\params]{\params \times \fibTotal_\initial}{\params}$
	and
	$\var[p]{\fibTotal}{\params}$
	in the slice-2-category $\Cat/\params$. In explicit components, $(p\times \radj{i})\circ \Delta=\langle p,\radj{i}\rangle$, and its left adjoint sends $(\param,E_0)$ to $\ladj{p}\param + iE_0$.
\end{lemma}
%
% \begin{figure} \label{fig:the_monad_inducing_diagram}
% \[\vxy{
% 	\fibTotal_\varnothing \pair i{\radj{i}} & \fibTotal
% 	\ar@<.5em>[d]^{\pairing{\radj{i}}{p}}
%     \ar@{}[d]|-{\dashv}
%     \ar@<-.5em>[r]_-{p}
% 	\ar@{}[r]|-{\perp}
% 	\ar@<.5em>@{<-}[r]^-{\ladj p}
% 	& \params \\
% 	& \clE_\varnothing\times\params \ar@/^/[ul]^{T^p}
%     \ar@<.5em>[u]^{\coPairSum{i}{\ladj{p}}}
% 	}\]
% \end{figure}
%
\begin{proof}
	Adjunctions compose, so \eqref{composite_adj} is a pair of adjoint functor between $\params\times \fibTotal_\initial$ and $\fibTotal$. Then, it is clear that $\langle p,\radj{i}\rangle$ is a $1$-cell in $\Cat/\params$. As for its left adjoint, we can calculate, given $(\param,E_0) \in \params \times \fibTotal_\initial$, that
	\[
		p(\ladj{p}\param + iE_0) \cong p(\ladj{p}\param) + p(iE_0) \cong \param + \initial \cong \param = \pi_\params(\param,E_0) \,.
	\]
	Here we used the assumption that $p$ preserves finite coproducts, the assumption that $p \circ \ladj{p} \cong 1_\params$ due to $\ladj{p}$ being full and faithful, and that $p\circ i$ is constantly initial due to $i$ being the inclusion of the initial fibre of $p$.
\end{proof}

\begin{remark}
	The importance of the condition on $T_\initial$ is that the initial fibre of the \textEM fibration of $T$ will be equivalent to the category $\carriers$ of carriers, which is precisely what allows us to recover the parametrized monad $T$ from the fibration $p^T$ of its algebras.

    In particular, the forgetful functor $V : \algTotal{\params}{\carriers} \too \carriers$ can be recovered, up to isomorphism, as the right adjoint $i^T_R : \algTotal{\params}{\carriers} \too \carriers$ of the inclusion $i^T$ of the initial fibre.
\end{remark}

We are now ready to give a description of the mapping from fibrations to parametrized monads.
\begin{construction}\label{constr_param_from_fib}
	Given a category $\params$ with an initial object and a fibration $p$ in $\niceFib{\params}$, we get a parametrized monad $T^p  : \params \to \Mnd{\fibTotal_\initial}$, for which the action of the underlying functor is
	\[
		T^p_\param (E_0) = \radj{i}(\initial_A + E_0) =  \radj{i}(\ladj{p}\param + iE_0) \,.
	\]
	This parametrized monad will belong to $\ParMndInit{\params}$, since due to  $\initial_\initial = \ladj{p}(\initial)$ we can see that  $T^p_\initial$ is isomorphic to the identity monad.
 
	The monad $T^p$ is induced by the adjunction of \autoref{lemma:_fibrations_to_adjunctions}.
    %The definition makes it obvious that $T^p_\varnothing\cong\id$, as $p_L\varnothing\cong\varnothing$.
    For future convenience, let us explicitly give the monad structure on $T^p_\param$:

    %, let's find the unit and the multiplication:
	\begin{itemize}
		\item the unit is obtained from the first coproduct injection as the composition
		      \[
			      \vxy{
			      \eta_\carrier : \carrier \ar[r]^{\eta_\carrier^{i}} & \radj{i}i\carrier \ar[r]^-{\radj{i}(\kappa')} & \TP \carrier
			      }
		      \]
		      where $\eta^{i}$ is the unit of the adjunction $i\dashv \radj{i}$ (which is invertible, as the left adjoint is fully faithful).
		\item multiplication is obtained from the counit $\varepsilon^{i}$ of $i\dashv \radj{i}$ and the `fold' map $\ladj{p}\param+\ladj{p}\param\to \ladj{p}\param$, \ie as the composition
		      \[\vxy{
			      \TP{(\TP \carrier)} \ar[d]^{\TP{(\varepsilon_X^{i } + i\carrier)}} \\
			      \radj{i}(\ladj{p}\param + (\ladj{p}\param + i\carrier)) \ar[d]^{\radj{i}(\nabla + i\carrier)} \\ \TP \carrier
			      }\]
	\end{itemize}
\end{construction}
For similarly convenient purposes, let's record the fundamental diagram the monad $T^p$ fits into:
\[\vxy{
	\fibTotal_\varnothing \pair i{\radj{i}} & \fibTotal
	\ar@<.5em>[d]^{\pairing{\radj{i}}{p}}
    \ar@{}[d]|-{\dashv}
    \ar@<-.5em>[r]_-{p}
	\ar@{}[r]|-{\perp}
	\ar@<.5em>@{<-}[r]^-{\ladj p}
	& \params \\
	& \clE_\varnothing\times\params \ar`l[l]_(.9){T^p}[ul]
    \ar@<.5em>[u]^{\coPairSum{i}{\ladj{p}}}
	}\]

We can finally state the theorem that allows us to measure how close a fibration is to being an \textEM fibration.
\begin{theorem} \label{theorem:_the_fibration-monad_adjunction}
	Given a category $\params$ admitting an initial object, there exists a 2-ad\-junc\-tion as in the left here,
	\[\label{pruned_adjunction}
 % \vcenter{
 \vxy[@R=1mm]{
 & & \fibTotal \ar[rr]^{\eta_p} \ar[ddr]_{p} && \algTotal[T_p]{\params}{\fibTotal_\initial}\ar[ddl]^{p^T}\\ 
 \niceFib{\params} \pair{p \mapsto T^p}{p^T \mapsfrom T} & \ParMndInit{\params} \\ 
 & && \params &
 } 
	\]
	making the unit a fibered functor, commuting with $p, p^T$ as in the right triangle. 

	Such unit is induced by the terminal object property of the \textEM construction.

	Moreover, the unit $\eta_p$ is an equivalence precisely when $p$ is the \textEM fibration of a parametrized monad $T$ for which $T_\initial$ is isomorphic to the identity monad.
\end{theorem}
\begin{proof}
  \autoref{constr_param_from_fib} yields a parametrized monad $T^p$ out of a pruned fibration, and this defines the left adjoint on objects;
%\ulo{The theorem seems to be lacking a proof. Prove that $\eta_p$ satisfies the universal property of an adjunction unit. (For a 2-adjunction we also need to check universality on the 2-dimensional level.) Or possibly, also produce the counit and show the triangle identities. (The counit could also be useful for showing that ParMnd -> Fib is full and faithful.)}
%\ulo{Check that a morphism of monads gives rise to a morphism of fibrations.}
the functoriality (on 1-cells) follows from the fact that given a morphism of fibrations $H : \clE \to \algTotal\params\carriers$ fibred over $\params$, and thus the functor between the fibers $H_\varnothing : \clE_\varnothing \to \algTotal\varnothing\carriers = \carriers$ (\ie the category of algebras of the identity monad), one can build the diagram
\[\label{magic_diagram}\vxy{
\params\times\clE_\varnothing \ar@{->}[rr]^{T^p} \ar@{->}[ddd]_{\id\times H_\varnothing} \ar@{->}[rd]_{p_L+i} &  & \clE_\varnothing \ar@{->}[ddd]^{H_\varnothing} \\
& \clE \ar@{->}[d]_H\ar@{}[u]|{\textsc{i}} \ar@{->}[ru]_{i_R} & \\
& \params\ltimes^T\carriers \ar@{}[d]|{\textsc{ii}}\ar@{->}[rd]^U \ultwocell<\omit>{\varphi} & \ultwocell<\omit>{\textsc{bc}} \\
\params\times\carriers \ar@{->}[rr]_T \ar@{->}[ru]^F &  & \carriers
}\]

An intrinsic choice of a functor isomorphic to $H_\varnothing$ is as the composition  
\[\vxy{\clE_\varnothing \ar[r]^-i & \clE \ar[r]^-H & \algTotal\params\carriers \ar[r]^-V & \carriers}\]
($i$ includes the fibre, while $V$ is the forgetful functor of \eqref{quaglia}, choosing the carrier of a parametrized algebra). 

More in detail:
\begin{itemize}
\item the right cell ``\textsc{bc}'' is filled by an isomorphism obtained from the fact that $\fibMor$ is a morphism of fibrations (and thus has components $H_A : \clE_A\to \algTotal A\carriers$ that satisfy the Beck-Chevalley condition), 
\item the left 2-cell is obtained from the freeness of the algebra, as follows 
  \[U*\varphi = UF(A,H_\varnothing X) \xto{a} H_\varnothing X \hookrightarrow Hp_LA + H_\varnothing X \cong UH(p_L+i)(A,X).\]
%  ($H_\varnothing X$ is a $UF(A,\firstblank)=T_A$-algebra because...\todo{})
\item the two triangles are strictly commutative, as they are the factorization of the monads $T$ and $T^p$. 
\end{itemize}
Now, the 2-cell so obtained splits, thanks to the definition of $H_\varnothing$, into smaller 2-cells each of which is commutative. Then, call $\theta$ the 2-cell so obtained; in order to prove that $(\id,H_\varnothing,\theta)$ is a morphism of parametrized monads (cf. \autoref{same_but_4monads}), we have to show that the two pasting diagrams of \eqref{par_mnd_axioms}
%\vspace{2mm}
%\[\vxy{
%    \ar@/^1.5pc/[rr]\ar[r]\ar[d] & .\ar[r]\ar[d] & \ar[d]\lltwocell<\omit>{<2>\mu} & . %\ar[d] \ar@/^1.5pc/[rr] & & . \ar[d]\\ 
%    . \ar[r] & . \ultwocell<\omit>{\theta}\ar[r] & \ultwocell<\omit>{\theta} . & .\ar[r] %\ar@/^1.5pc/[rr] & .\ar[r] & . \lltwocell<\omit>{<2>\mu}\ulltwocell<\omit>{<3>\theta}
%}\]
are commutative. This can be checked by hand.% follows from\todo{bla} 

Now, on 2-cells, a vertical natural transformation $\xymatrix{X \rtwocell^H_K{\alpha} & Y}$ must induce a 2-cell between parametrized monads; $\alpha$ certainly induces a natural transformation $\alpha_\varnothing : H_\varnothing \To K_\varnothing$, and it can be checked that it satisfies the commutativities of \autoref{same_but_4monads} required to be a monad 2-cell.

In the opposite direction, functoriality (on 1- and 2-cells) of the assignment $(\params,\carriers,T)\mapsto \algTotal[T]\params\carriers$ follows from \autoref{yoink_a_2_fibration}, adapted to monads.

Now, let's turn to the proof that the canonical map $\eta_p$ in
\[\vxy{
	\fibTotal
	\ar[rr]^{\eta_p}
	\ar[dr]_{\pairing{p}{\radj{i}}}
	& & \algTotal[T_p]{\params}{\initialFibre{\fibTotal}}
	\ar[dl]^{\pairing{p^T}{\radj{i^T}}} \\
	& \params \times \initialFibre{\fibTotal} &
}\]
obtained as comparison functor into the \textEM category $\params \times \initialFibre{\fibTotal}$ of the monad $\pairing{p}{\radj{i}} \circ \coPairSum{\ladj{p}}{i}$ is a reflection; this will conclude the proof, as the counit of the adjunction consists of the identification 
\[T \cong \text{the monad obtained from $p^T$ using \eqref{constr_param_from_fib}}\]
and thus is invertible and the adjunction identities simplify into the requirement that every span of morphisms in $\Cat/\clA$ like 
\[\vxy{
\clE\ar[d]_{\eta_p}\ar[r]^-K & \algTotal[S]{\params}{\carriers}\\ 
\algTotal[T_p]{\params}{\initialFibre{\fibTotal}}\ar@{.>}[ur]_{\bar K}
}\]
(where $S : \params\times\carriers\to\carriers$ is another pruned monad) has a unique `extension' $\bar K$ along $\eta_p$. To prove this, one argues building up a diagram similar to \eqref{magic_diagram}: a functor $\algTotal[T_p]{\params}{\initialFibre{\fibTotal}} \to \algTotal[S]{\params}{\carriers}$ exists by functoriality of $\algTotal{-}{-}$ construction if we build a morphism $(\params,\initialFibre{\fibTotal},T_p)\to (\params,\carriers,S)$ in $\algTotal[\ell]{\Cat}{\Cat}$, which is realized by $(\id[\clA],K_\initial;\theta)$ where $\theta$ is a natural transformation 
\[\vxy{
\params\times\initialFibre{\fibTotal}\ar[d]_{\id\times K_\initial}\ar[r]^-{T^p} & \initialFibre{\fibTotal} \ar[d]^{K_\initial}\\ 
\params\times\carriers \ar[r]_-S & \carriers\ultwocell<\omit>{\theta}
}\]
obtained filling a square in all similar to \eqref{magic_diagram}. This concludes the proof.
\end{proof}

\begin{remark}
	The failure of the unit $\eta_p$ in \autoref{theorem:_the_fibration-monad_adjunction} to be an equivalence measures the failure of $p$ to be an \textEM fibration for a $T$ with $T_\initial \cong \id$. One could equivalently measure the failure of the functor $\langle p,\radj{i}\rangle  : \fibTotal \to \params \times \fibTotal_\initial$ to be monadic.
\end{remark}
The result in \autoref{theorem:_the_fibration-monad_adjunction} lends itself to a straightforward dualization: an opfibration $\var[p]\clE\clB$ is pruned if the terminal fibre $\clE_1$ has an initial and terminal object (so $\clE_1$ is reflective with reflector $\ladj i$), and if $p$ has a right adjoint $\radj p$ and preserves products. Similarly, a pruned comonad $S : \params\times\carriers\too\carriers$ is such that $S_1$ is the identity functor. Then,
\begin{theorem}
	There is an adjunction
	\[\label{copruned_adjunction}
		\vxy{ \niceOpFib{\params} \oppair{p^S \mapsfrom S}{p \mapsto S^p}& \ParCoMndInit{\params} }
	\]
	between pruned opfibrations and pruned comonads, sliced over $\params$, and such that a fibration is the co\textEM opfibration of a parametrized pruned comonad if and only if the counit of said adjunction is invertible. The comonad is obtained as the comonad associated to the adjunction $\langle p,\ladj i\rangle\dashv (p_R\times i)$
\end{theorem}
\subsection{Examples of the fibration to parametrized monad construction}

Instances of our \autoref{theorem:_the_fibration-monad_adjunction} can be found in the literature in a number of different contexts.

\begin{example}[Artin glueing]
	Given a finite limit preserving functor $F  : \clE\too \clE'$ between elementary toposes, the \emph{Artin glueing} \cite{Wraith1974} of $\clE,\clE'$ along $F$ is defined as the comma category $\clE'/ F$. This arises as the total category of the co\textEM opfibration of the parametrized comonad $\clE \to \coMnd{\clE'}$, $E \mapsto  FE\times \firstblank$.

	In \cite{Faul2023}, the authors consider `adjoint split extensions of toposes', which, they show, always arise from the Artin glueing construction. An adjoint split extension of toposes is a diagram of toposes and functors, where one of the functors can be shown to be an opfibration. An application of the dual of \autoref{theorem:_the_fibration-monad_adjunction} to that opfibration yields precisely a parametrized comonad of the form $E \mapsto FE\times \firstblank$, and the unit $\eta_p$ will be an equivalence.

	This is an instance of a setting where the category $\algTotal{\clE}{\clE'}$ of coalgebras truly plays the role of a semidirect product for a class of categories.
\end{example}

Another class of applications of \autoref{theorem:_the_fibration-monad_adjunction} can be found where for a fibration $p$, depending on a parameter, the unit $\eta_p$ need not be an equivalence, but one is particularly interested in those parameters for which $p$ is is an equivalence or satisfies some other condition.
As an instance of this, we have the following example from categorical algebra. An account on this which is close to ours can be found in
\cite{BJK:internal_object_actions}.
\begin{example}[The fibration of points]
	The \emph{freestanding split epi} is the category generated by the graph $r  : 0 \rightleftarrows 1 : s$,
	subject to the relations requiring that $r\circ s=\id[1]$ (and thus $s\circ r$ is an idempotent). A \emph{point} in a category $\params$ is a functor $P$ from the freestanding split epi to $\params$. Letting $\pointsTotal{\params}$ denote the category of points in $\params$, the evaluation functor $P \mapsto P1$ yields a fibration $p  : \pointsTotal{\params} \to \params$ as long as $\params$ has pullbacks along split epimorphisms. This is known as the \emph{fibration of points} of $\params$ and its properties are central to a lot of categorical algebra.

	If the category $\params$ has a zero object and finite coproducts, then applying \autoref{theorem:_the_fibration-monad_adjunction} to the fibration of points yields a parametrized monad $T^p  : \params \to \Mnd{\params}$, commonly denoted by $\param \flat \carrier \coloneqq T^p_\param\carrier$. The left adjoint $\algTotal{\params}{\params} \to \pointsTotal{\params}$ of $\eta_p$ exists if $\params$ has reflexive coequalizers, and is interpreted as the functor that maps an action to its semidirect product.

	If $\eta_p$ is an equivalence, then $\params$ is said to be \emph{a category with semidirect products} \cite{Bourn1998}. If $\eta_p$ merely reflects isomorphisms, then the category $\params$ is said to be \emph{protomodular}. Note that protomodularity is more commonly, but equivalently, defined as reflection of isomorphisms by the functor $(p,i^R)  : \pointsTotal{\params} \to \params \times \params$.

\end{example}

\begin{example}[Fibrational models of Martin-Löf type theory]
	In the context of fibrational models of Martin-Löf type theory, if we have a \emph{(split) comprehension category with unit} \cite[Def. 10.4.7]{CLTT} which is moreover a bifibration (\ie, it models dependent sums along all morphisms in its base category), then we have a bifibration $p  : \fibTotal \to \params$ and adjunctions $p \dashv 1 \dashv \{\firstblank\}$. Applying the dual of Theo\-rem~\ref{theorem:_the_fibration-monad_adjunction} to $p$ viewed as an opfibration yields the parametrized comonad
	\[
		\params \to \coMnd{\params}\,, \quad A \mapsto A \times \firstblank
	\]
	of Example~\ref{fib_reg_rep},
	whose opfibration of co\textEM coalgebras is the codomain opfibration $\params^\rightarrow \to \params$.
	The unit $\eta_p  : \fibTotal \to \params^\rightarrow$ is the full and faithful functor (called the \emph{full comprehension category} in models of Martin-Löf type theory \cite[Def. 10.4.2]{CLTT}) that exhibits $\fibTotal$ as the category of display maps.
\end{example}